\documentclass[draft,10pt,reqno]{amsart}
\usepackage{amsbsy,amssymb,amsfonts,graphicx}

\allowdisplaybreaks

\DeclareMathOperator{\supp}{supp}

\def\squarebox#1{\hbox to #1{\hfill\vbox to #1{\vfill}}}
\newtheorem{THEOREM}{Theorem}
                        {\end{THEOREM}}
\newtheorem{COROLLARY}[THEOREM]{Corollary}
                          {\end{COROLLARY}}
\newtheorem{LEMMA}{Lemma}
                      {\end{LEMMA}}
\newtheorem{CLAIM}{Claim}
                      {\end{CLAIM}}
\newtheorem{EXAMPLE}{Example}
                      {\end{EXAMPLE}}
                      {}
\newtheorem{Theorem}{Theorem}
\newtheorem{Lemma}{Lemma}[section]

\newtheorem{Def}{Definition}
\newtheorem{Remark}[Lemma]{Remark}
\newtheorem{Corollary}[Lemma]{Corollary}

\begin{document}
\setcounter{page}{1}
\pagestyle{plain}
\newcounter{C1}
\renewcommand{\thefootnote}{\fnsymbol{footnote}}

\title{The thin film equation with backwards second order diffusion}
 \author{Amy Novick-Cohen\\
Department of Mathematics\\ Technion-IIT, Haifa 32000, Israel\\ amync@tx.technion.ac.il\\
 \and\\
Andrey Shishkov\\ Institute of Applied Mathematics and
Mechanics,\\  83114 Donetsk, Ukraine\\ shishkov@iamm.ac.donetsk.ua}

\maketitle

\begin{center}{\it{Dedicated to  Roberta Dal Passo
(1956-2007)}}\end{center}

\begin{abstract}
In this paper, we focus on the thin film equation with  lower order
"backwards" diffusion  which can describe, for example,
 the evolution of thin viscous films in the presence of gravity and thermo-capillary effects, or the thin film
 equation with a ''porous media cutoff" of van der Waals forces.
We treat in detail the equation
$$u_t + \{u^n(u_{xxx} + \nu u^{m-n}u_x -A u^{M-n} u_x)\}_x=0,$$
where $\nu=\pm 1,$ $n>0,$ $M>m,$ and  $A \ge 0.$ Global existence of
weak nonnegative solutions is proven when $ m-n> -2$ and $A>0$ or $\nu=-1,$ and when
$-2< m-n<2,$ $A=0,$  $\nu=1.$ From the weak solutions, we get strong
entropy solutions under the additional  constraint that $m-n> -{3}/{2}$ if $\nu=1.$ A
local energy estimate is obtained when $2 \le n<3 $ under some
additional restrictions. Finite speed of propagation is proven when $m>n/2,$ for
the case  of ''strong slippage,"  $0<n<2,$  when $\nu=1$
based on local entropy estimates, and for the case  of ''weak
slippage," $2 \le n<3,$ when $\nu=\pm 1$ based on local entropy and
energy estimates.
 \end{abstract}


\newpage

\section{Introduction}
\setcounter{equation}{0}

The thin film equation \cite{SR86}
\begin{equation} \label{thinfilm}
u_t + \{u^n(u_{xxx})\}_x=0, \quad n>0,
\end{equation}
with $n=3$ models the dynamics of thin viscous films with no slip boundary conditions, and
with $n=1$ it models Hele-Shaw flow.
Often (\ref{thinfilm}) needs to be augmented with various lower order terms
in order to take into account the presence of  additional physical
effects, and certain such equations shall be considered here. See  \cite{ODB97} for a survey and review. Some of the physical systems
which are  accommodated by the analysis in this paper include:


\smallskip\par\noindent$(i)$ the evolution of thin viscous
films in the presence of gravity and thermo-capillary effects
\begin{equation} \label{gthermoc}
u_t+ \{u^n(u_{xxx} + u^{m-n}u_x - A u^{M-n} u_x)\}_x=0,
\end{equation}
where $m,$ $n,$ $M,$ $A$ are constants such that $A \ge 0,$
$n>0$,$\,m<M,$
 as well as the  more accurate variant of (\ref{gthermoc}) given by
\begin{equation} \label{gthermoc2}
u_t + \{u^n(u_{xxx} + h'(u)u_x)\}_x=0,
\end{equation}
where $h'(u)= - \overline{\nu} G+ \frac{B_1}{u(1 + B_2 u)^{2}},$ $G, B_1,  B_2$ are positive constants, $\overline{\nu}=\pm1,$ and where $\overline{\nu}=+1(-1)$ represents stabilizing
(destabilizing) gravitational forces
\cite{ODB97,TK04,OR92}. Equation (\ref{gthermoc2}) with $A=0$,
$$u_t + \{u^n(u_{xxx} + u^{m-n}u_x)\}_x=0,$$
models  thin  films with
 thermo-capillary effects but  without gravitational effects when $m-n=-1$ \cite{Ehrhard}, or with destabilizing gravitational effects but without capillary effects when $m-n=0$. The value of $n$ in (\ref{gthermoc}), (\ref{gthermoc2}) reflects the assumptions on the
 slip conditions at the interface of the thin film with the underlying substrate, with  $n=3$ modeling no slip
 and  $0<n<3$ modeling various types of slip.

\smallskip\par\noindent$(ii)$ the equation
\begin{equation}\label{attractive}
u_t + \{u^n (u_{xxx} + h'(u) u_x)\}_x=0,
\end{equation}
describes (i) the evolution of a thin viscous film in the
presence of attractive polar forces if $h(u)=-b_1 e^{-u/b_2}$ where $b_1,$
 $b_2$ are positive constants, or (ii) the evolution of a thin
viscous film in the presence of attractive van der Waals forces if
$h(u)=B u^{-b},$ where $B<0$ is a (negative) Hamaker constant
and $b$ is a positive constant. More generally, if $\lim_{u \downarrow 0} h'(u)>0$ ($<0$), $h(u)$ is said
to represent limiting attractive (repulsive) forces. Equations of the form (\ref{gthermoc}), (\ref{attractive}) can  represent a combination of
attractive and repulsive forces, and the limiting power $m=\lim_{u \downarrow 0} u \, h''(u)/h'(u)$ can assume a  range of both positive and negative values in modeling various forces such as
polar forces, van der Waals forces, as well as ''porous media cutoff" of van der Waals forces. See \cite{MNC1,ODB97,BP94,Derjaguin,Eggers}.

\smallskip\par\noindent$(iii)$
the Hocherman-Rosenau equation \cite{HochermanR}
\begin{equation}\label{hochermanR}
u_t + \{f(u) u_{xxx} + g(u) u_x)\}_x=0,
\end{equation}
 was proposed as a generalization of the cylindrical Kuramoto-Sivashinsky equation, which  models low Reynolds number two phase cylindrical flows.
  Equation (\ref{hochermanR}) with $f(u),$ $g(u)>0$ has been used as a prototype equation for
studying the relative strength of the second and fourth order terms in determining criteria for  blow up \cite{HochermanR,BP3}.
Setting $f(u)=u^n$ and $g(u)=u^m$ or $g=h'(u)$ with $h'(u)>0$ in (\ref{hochermanR})
yields   (\ref{attractive}) with ''limiting attractive forces."

\bigskip
 To
understand these various models, we shall focus on the
equation
\begin{equation} \label{G}
u_t + \{u^n(u_{xxx} + \nu u^{m-n}u_x -A u^{M-n} u_x)\}_x=0,
\end{equation}
where $\nu=\pm 1,$ and where $n,$ $m,$ $M,$ $A$ are constants satisfying
$A \ge 0,$ $n>0,$ $M>m,$ with some further restrictions to be
imposed in the sequel.   All of the  examples outlined
   above may be written in the form (\ref{G}), except for
(\ref{gthermoc2}), (\ref{attractive}\,i) in which
\begin{equation} \label{1.7}
 h'(u)=u(1 + Bu)^{-2} \hbox{\, or \,}
h'(u)=\frac{b_1}{b_2}e^{-u/b_2},
\end{equation}
respectively. Noting that
\begin{equation} \label{boundsa}
 0<\frac{u}{(1+Bu)^2}< \frac{1}{B^2 u},
\end{equation}
\begin{equation}\label{boundsb}
 0<\frac{b_1}{b_2}e^{-u/b_2}<\frac{b_1}{b_2},
\end{equation}
for $u \ge 0,$ we see that the examples in (\ref{1.7}) have upper bounds of the
form $h'(u)=\nu u^{m-n} -A u^{M-n}$ with $\nu=1,$ $A=0,$ and
$m-n=-1$ and $m-n=0$ respectively, which facilitate the analysis
of (\ref{gthermoc2}), (\ref{attractive}\,i). Our treatment of
(\ref{G}) can be generalized to encompass  (\ref{1.7})
 as well; some remarks in this direction are included
 within the text.

In terms of the physical systems being modeled, it is reasonable
to assume zero contact angle, $u_x=0,$ and zero flux, $u^n(u_{xxx}
+ \nu u^{m-n}u_x -A u^{M-n} u_x)=0,$ along the external boundaries
of the domain, and thus  that
\begin{equation} \label{Neumanbc}
u_x=u_{xxx}=0 \hbox{\, if \,} u>0 \hbox{\, and \,} x=\pm a. \end{equation}  While these
boundary conditions are adopted here,  other boundary conditions are possible to consider.
The zero contact angle condition reflects the physical assumption that the viscous liquid film completely wets the underlying
substrate; non-zero contact angle conditions, while physical, have proven up to now to be difficult to implement in dynamical problem formulations \cite{Otto}.
A simple alternative to (\ref{Neumanbc}) is to impose periodic
boundary conditions, as we shall amplify.

The term "$\{u^n(\nu u^{m-n}u_x)\}_x$," with $\nu=+1$ in (\ref{G}),
is often referred to as a "backwards diffusion"  term, since if one
considers dynamics dominated by this term alone
$$u_t + \{u^n(u^{m-n} u_x) \}_x=0,$$
and one linearizes about a uniform positive  state, then the
resultant dynamics is given by the
 backwards (ill-posed) diffusion equation. Similarly, if $\nu=-1,$ the term $"\{u^n(\nu u^{m-n}u_x)\}_x"$  in (\ref{G})
 is often referred to as "forward diffusion," for obvious reasons \cite{CHE}.
We shall often refer to  equation (\ref{G}) with $\nu=+1$ as the
"unstable case" and to equation  (\ref{G}) with $\nu=-1$ as the
"stable case," since in the context of thin films, (\ref{G}) with
$\nu=+1$ models limiting attractive  forces which are destabilizing ("long wavelength unstable" in the terminology of \cite{BP98})
 and (\ref{G}) with $\nu=-1$ models limiting
repulsive  forces which are stabilizing \cite{ODB97}.

The degenerate Cahn-Hilliard  equation
\begin{equation} \label{DCH}
u_t+ \{u(1-u)[ -\ln u + \ln(1-u) + \alpha u +\epsilon^2 u_{xx}]_x\}_x=0,\end{equation}
whose solutions satisfy $0 \le u \le 1$ \cite{EG,CHE}, can also be said to be of the form (\ref{G})
with $\nu=-1, n=1, m=0                                 $ for $u$ near  $0$ or $1.$
Equations similar to (\ref{attractive}), (\ref{G}) also arise in modeling structure formation and the
dynamics of biofilms \cite{KD05}, as well as in modeling {the}  dislocation density in plasticity theory \cite{GrunP}.

Equation (\ref{G})  in the presence of forward
or stabilizing diffusion
\begin{equation} \label{forwardsintro}
u_t+ \{u^n(u_{xxx} - u^{m-n} u_x)\}_x=0,
\end{equation}
has  been studied somewhat more thoroughly than the backwards or unstable variant
\begin{equation} \label{backwardsintro}
u_t+ \{u^n(u_{xxx} + u^{m-n} u_x)\}_x=0.
\end{equation}
This is perhaps not  surprising, as the behavior of the thin film equation   changes qualitatively less upon adding a stabilizing term as opposed to adding
a destabilizing term, and hence  the analytical tools used in studying (\ref{thinfilm}) may be more readily adapted to its analysis.

Since (\ref{thinfilm}) is designed to model the evolution of this viscous films,  nonnegative initial
date should yield nonnegative solutions. Compactly supported initial data whose support which may spread or shrink are also physically relevant to consider.
A pioneering step in the analysis of (\ref{thinfilm}) was the proof  of the existence of weak nonnegative solutions for $1<n<3$ by Bernis \& Friedman \cite{BF90}
for  nonnegative $H^1$ initial data, which relied
   on energy and entropy estimates.
 The analysis in \cite{BF90} was  extended by Bernis \cite{B0} to the interval $0<n<3$.
  In \cite{BPW}, existence (nonexistence) was demonstrated of compactly supported spreading source type solutions to (\ref{thinfilm}) for $0<n<3$ ($n \ge 3$), and at the contact line, the solutions
were seen to behave like $(x_0-x)^2$ for $0<n<3/2$ and like $(x_0-x)^{3/n}$ for ${3}/{2}<n<3.$
By developing refined entropy estimates, the existence of strong ($C^1$ for a.e. $t>0$) nonnegative solutions was demonstrated
 by  Beretta, Bertsch  \& Dal Passo \cite{BBD95}
and by Bertozzi \& Pugh \cite{CPAM96}, and the solutions  were seen to possess the regularity of the source type solutions at the contact line. Positivity properties were seen to depend strongly
on the value of $n$ \cite{BF90,BBD95}, with "touchdown" being possible for small values of $n$ \cite{BK,BBD95}. Solutions were proven to become positive in finite time, and to converge to
their mean as $t \rightarrow \infty$ \cite{BBD95,CPAM96}.  The only nonconstant steady states  for (\ref{thinfilm}) are of the form  $(x-a)^{+}(b-x)^{+}$ for $a<b$, which do  not possess the regularity of the
   strong solutions \cite{BBD95,CPAM96}. Finite speed of propagation of the support of strong solutions for $0<n<3$ was proven in \cite{B1,B2}.
Some of the details of the various methodologies will be amplified further as we present our results.
Since these basic studies, the analysis has been developed considerably
to encompass, for example, the Cauchy problem for measured valued initial data \cite{DG}, waiting time phenomena \cite{DGGw},  higher dimensional studies of (\ref{thinfilm}) \cite{DGG,BDGG},
as well as various numerical schemes \cite{GrunR,Grun}.

\bigskip
With regard to (\ref{forwardsintro}),  existence of nonnegative
distributional solutions was proven in \cite{BP94}  for $0<n$ and
$0<m<1$. These solutions were shown {to} approach their mean as $t
\rightarrow \infty$, becoming positive in finite time. Formal
asymptotics and numerics were used to suggest the existence of
advancing fronts for $n \ge 3,$ as well as for $0<n<3$. Existence
of nonnegative solutions for $\Omega \subset \mathbb{R}^N,$
$N=1,2,3$ was then demonstrated for two special cases of
(\ref{forwardsintro}); namely, for the model for defects in
plasticity theory mentioned earlier \cite{GrunP}, and for
(\ref{DCH}), the Cahn-Hilliard  equation  \cite{EG}. In
\cite{DGS01}, an existence theory for (\ref{forwardsintro}) was
presented which encompassed the cases $n> 1/8,$ $m>-1,$ $N \le 3$
({with} $n <4$ if $N=3$), with the case $0<n<1/8$ being treatable
by adjusting the definition of the solution.   The existence
results   in \cite{BP94,DGS01} build on the results in
\cite{BF90,BBD95,CPAM96}, and in \cite{DGS01} it builds on
\cite{GrunP} as well. In particular, use is made of augmented
entropy estimates \cite{BBD95,CPAM96}, which indicate that the
contact angle is zero for the solutions obtained if $0<n<3$ and
$n-m>-2$. The constraint $n-m>-2$ allows the energy to be bounded
from below {for arbitrary nonnegative initial data $u_0 \in H^1$.} Regularity at the contact line  was confirmed  via
formal asymptotic arguments in \cite{BP94}. In \cite{DGS01},
finite speed of propagation is demonstrated for $m>0,$ $1/8<n<2,$
and estimates are obtained for the speed of propagation; infinite
speed of propagation is proven for $m<0,$ {$-2<m-n<-3/2$.}

In \cite{B97}, self-similar spreading source type solutions were shown to exist for (\ref{forwardsintro}) when $m-n=2,$ if $0<n<3,$  and not
to exist if $n \ge 3.$ The asymptotics of these solutions at the contact line was seen to match that of the source type solutions found  for (\ref{thinfilm}) in \cite{BPW}.
In terms of steady states, in \cite{LP00}
nonconstant positive periodic steady states and zero contact angle steady states are shown not to exist, though touchdown steady states with non-zero contact angle are shown to exist if $m-n>-2.$

In \cite{GrunR, Grun}, existence of positive solutions from {nonnegative} initial data was demonstrated via construction of a numerical scheme,
for (\ref{G}) with $\nu=-1$ {and $m-n<-2$.}
The inclusion of the term  $-A u^M$ in (\ref{G})
 can be expected to enhance the rate of decay of the solution to its mean, and should not have any major effect on contact angle and propagation properties.
In our analysis in the present  paper, we {closed gaps in} the parameter range for (\ref{forwardsintro}) and
(\ref{G}) with $\nu=-1,$
both in terms of existence and  the finite speed of propagation property. Moreover, local energy and entropy estimates are obtained.

With regard to (\ref{backwardsintro}),  less attention has been paid to the various existence, regularity, and finite speed of propagation properties for (\ref{backwardsintro})  than for (\ref{forwardsintro}), but there has  been longstanding interest in the qualitative predictions of (\ref{backwardsintro}),
in particular in regard to the possibilities of rupture \cite{SR86,WD82}  and blow up  \cite{HochermanR}. There has been considerable interest in various self-similar and steady state solutions for
(\ref{backwardsintro}), and their stability \cite{Sl07}.

In terms of existence, regularity, and finite speed of propagation properties for (\ref{backwardsintro}), the results up to now may be summarized as follows. In
 \cite{BP98} it is demonstrated that if $0<n<3$  and $n \le m<n+2,$  then there exist globally bounded nonnegative weak solutions to (\ref{backwardsintro})
which possess the finite speed of propagation (FSP) property, and have the same  regularity at  the contact line as was found in \cite{BPW} for (\ref{thinfilm}).
We remark  that formal asymptotics  developed in \cite{BP94} for (\ref{forwardsintro}) are equally valid for (\ref{backwardsintro}).
 The methodology in \cite{BP98}  relies on \cite{BF90,BBD95,CPAM96,BP94}, incorporating a "disjoining pressure" {potential} arising from the lower order terms  \cite{Derjaguin} into the energy used in the basic energy estimate.
 Numerical evidence is given  that blow up may occur if $m \ge n+2$.
 The case $n=1,$  $m \ge 3=n+2$ is considered in \cite{BP3}, and it is demonstrated that for compactly supported nonnegative initial data whose initial energy is negative,
 there exists a solution  which blows up in finite time. For the critical case, $n=1,$ $m=3,$ this implies blow up if the initial mass is sufficiently large \cite{WBB04}.
It {was} conjectured in \cite{BP98,BP3} that $m-n\ge 0$ constitutes a necessary condition for well posed dynamics.
The existence proof in \cite{Grun} mentioned earlier in the context of (\ref{forwardsintro}), includes the case $\nu=1,$ i.e.  (\ref{backwardsintro}) with $n>0$ and $m<n-1$, for positive initial data.
The inclusion of a term of the form  $-A u^M$ in (\ref{G}) when $\nu=1$ should eliminate the possibility of blow up, and
  should not have any major effect on contact angle and propagation properties.

The set of steady states and self-similar solutions is much richer for (\ref{backwardsintro}) than for (\ref{forwardsintro}).
For (\ref{backwardsintro}), positive periodic steady states exist, as do compactly supported "touchdown"  steady state solutions with  zero as well as with nonzero contact angles \cite{LP00}.
A study of steady states and  stability based energetic criteria  was undertaken in \cite{LPJDE} for $m-n \in [1,2).$
It has been demonstrated \cite{LP00,LPARMA,LPelectronic,Sl07} that steady states with zero contact angle exist for all
$0<n<3,$ $n\le m,$ which are stable if $m \le n+2,$ $0<n \le 2,$ marginally stable if $m \le n+2,$ $2<n<3,$ and unstable otherwise.
In \cite{BGW} for $\nu=1$, $n=3$, $m=-1$,  steady states are seen to converge to a $\delta$-distribution in the limit in which repulsive forces are neglected;
these results should perhaps  be compared with the non-single-valued profiles
seen in \cite{MNC1} which result when the term $u_{xx}$ is replaced by the mean curvature, a correction which becomes important in the singular limit.
In terms of self-similar solutions, both spreading and blow up self-similar solutions are possible. Rupture self-similar solutions have also been observed and studied \cite{ZLister,WB}. If $m=n+2$,  there  exist spreading self-similar solutions  if $0<n<3,$
  and none  if $n \ge 3$ \cite{B97}. Blow up self-similar solutions were considered in \cite{SP05}, with existence being demonstrated for $0<n<3/2$ and non-existence for $n\ge 3/2.$
   For a discussion of the stability
  of self-similar solutions, see \cite{WBB04,Sl07}.

\par\smallskip\noindent
The focus of the present paper, however, is not on blow up, steady
state solutions, and self-similar behavior, but rather on
conditions that guarantee existence, regularity, and finite speed
of propagation,  completing and enhancing what was
previously know, in order to construct a framework in which
understand, for example, the transition between rupture,
positivity and touchdown properties, global existence, and blow
up. Roughly speaking, with regard to existence  for $\nu=-1,$ $0<n<3,$ there {had been} a gap for
{$n-2 \le m \le -1;$} and now we have  existence of both weak and strong
solutions {for $n-2 < m \le -1,$} {as well as for $n-2=m$ for a constrained set of initial conditions.} For $\nu=1$, $0<n<3,$ {previous existence results had
required $m \ge n$,} and now we have existence of weak solutions in
{the interval $n-2<m<n$,} as well as strong {energy/entropy} solutions in the subinterval
$n-3/2<m<n.$ Note also that for $\nu=\pm 1,$ $0<n<3,$ in the
interval $n-2<m \le n-1,$ we do not require that the initial data
be strictly positive {\cite{BGW}}.
In terms of FSP, for $\nu=-1,$ $0<n<3,$ {there had been a gap at $0<n \le 1/8$   which has now been filled for $0<m<n+2$,
as well as a gap at $2
\le n<3,$ which has now been filled for $n/2<m<n$.} For $ \nu =1$,
$0<n<3,$ there had been a gap at $0<m<n,$ which has now been filled
for $n/2<m<n$.

What physics pertains to the interval $\nu=- 1,$ $0<n<3,$
{$n-2 \le m \le -1$} or $\nu=1,$ $0<n<3$, {$n-2 \le m<n$?}
 The thin film equation with $\nu=1,$
$0<n<3,$ $-2<m-n<0$ can reflect a "porous media" cutoff of
attractive (or repulsive if $\nu=-1$) van der Waals forces
\cite{ODB97}, thin films under the influence of thermocapillary effects
  with $n=2$ or $3$, $m-n=-1,$ $\nu=1$,
\cite{Ehrhard} (with other values of $n,$ $0<n<3,$ also being possible to
consider), or a restricted Hocherman-Rosenau equation
\cite{HochermanR}. If one takes $\nu=1,$ $n=2$ or $3$, $m \le -1$,
which can model, for example, the effects of a Lennard-Jones potential or a thin
film on a layered solid substrate in the limit in which the
limiting repulsive forces are neglected,
 the dynamics are known to lead to rupture \cite{WD82,ZLister,WB}.

\par\bigskip\noindent
 As a first step in this direction (enhancing existence), the
existence of weak nonnegative solutions (see Definition
\ref{weaksolution}) is demonstrated in $\S 2$. This is accomplished
by means of the basic energy estimate \cite{BF90,BBD95,CPAM96}
\begin{equation} \label{globalenergyintro}
\sup_{0 \le t \le T} \int_{\Omega} u_x^2 \, dx+ \int_0^T \int_{\Omega}
u^n (u_{xxx} +\nu u^{m-n}u_x -A u^{M-n}u_x)^2 \, dx \, dt \le C,
\end{equation}
where $\Omega=(-a, a),$  $a \in (0, \, \infty)$ is arbitrary,    $C$ depends only on the problem
parameters and the initial conditions, and $0 < T < \infty.$ The estimate
(\ref{globalenergyintro})
  holds
 for   $0<n,$ $-2<m-n,$
$m<M,$ with the additional constraint that  $m-n<2$ if $\nu=1$ and  $A=0$. Additionally,
an entropy estimate  \cite{BF90,BBD95,CPAM96} is obtained, using  Gronwall's
inequality  for regularized solutions. The use here of
Gronwall's inequality, which explicitly depends on the
regularization parameter,  differs somewhat from elsewhere {\cite{GrunR}},  allowing us to control the lower order, possibly singular,
forcing terms.
 These estimates, together with mass conservation,
\begin{equation} \label{massconservation}
\int_{\Omega} u(x, \, t) \,dx=\int_{\Omega} u(x, \, 0) \, dx,
\end{equation}
 imply global bounds from which existence of
weak nonnegative solutions based on uniform H\"{o}lder continuity and continuation arguments. See Bernis \& Friedman \cite{BF90},
Giacomelli \cite{G99}.

\smallskip
To obtain the existence of a strong (${\mathcal{C}}^1(\Omega)$ for
a.e. $t>0$) solution, a local entropy estimate is  derived in $\S
3,$ following \cite{BBD95}. For  $\nu=1,$ the additional constraint
$m-n>-\frac{3}{2}$ is imposed and the local entropy estimate
may be written as
\begin{multline} \label{localentropyintro}
\frac1{\alpha(\alpha+1)}\int_{\Omega} \zeta^4 u^{1+\alpha}(x, \,
T) \, dx +  A \int_{Q_T} \zeta^4 u^{\alpha + M-1} u_x^2 \, dx\, dt+ \\ c_{1}
\Biggl[ \int_{P} \zeta^4 {u}^{\alpha + n -2 \gamma +1}
(u^{\gamma})_{xx}^2\, dx\, dt  + \int_{Q_T} \zeta^4
{u}^{\alpha +n-3} {u}_x^4\, dx\, dt \Biggr] \le\\
 c_{2} \int_{Q_T}( |\zeta_x|^4 + |\zeta \zeta_{xx}|^2) u^{n+\alpha
+1}\, dx dt + c_{3} \int_{Q_T} |(\zeta^3 \zeta_x )_x| {u}^{\alpha + m + 1}\, dx\,dt+\\
c_{4} \int_{Q_T} \zeta^4 u^{\alpha + 2m -n +1}\,dx\, dt+
\frac1{\alpha(\alpha+1)}\int_{\Omega} \zeta^4 {u_0}^{\alpha+1} \,
dx,
\end{multline}
where $Q_T=\Omega \times (0, \, T),$  $P=\overline{\Omega \times (0, \, \infty)} \setminus \{ u=0 \hbox{\, or \,} t=0 \}$, and which holds for certain $\alpha \in (\max\{-2m+n-1,\,
-m-1\},\,2-n)\setminus \{0,\, -1\}$ and for $\gamma$ satisfying
(\ref{a2.6}).  For  $\nu=-1,$ a similar estimate is obtained
without additional restrictions. For both  $\nu=\pm1,$ the
local entropy estimate implies the global entropy estimate
\begin{equation} \label{globalentropyintro}
c_5 \int_{Q_T} u^{\alpha + n -3} u_x^4 \, dx\, dt + c_6 \int_{Q_T}
u^{\alpha + n -1} u_{xx}^2 \, dx\, dt \le c_7\, T, \quad 0<T,
\end{equation}
  which holds also for (\ref{thinfilm}), and which implies
  strong solutions  and certain positivity properties \cite{BBD95,CPAM96}.
We also present a  refinement of Theorem 3.1 from \cite{BBD95}
(see {Corollary \ref{3.4}} in $\S 3$), which clarifies the set of $\beta$ for
which $C^1([-a,\,a])$ regularity for almost every $t>0$ is implied
for $u^{1/\beta}(\cdot,\, t)$  based on the local entropy estimates and the properties of the initial data. As in \cite{GS}, a
local energy estimate is derived for $2 \le n <3,$ under
the additional restrictions that $m>({2n} - {2})/{3}$ if $2 \le n
< {5}/{2}$ and $m>n-{3}/{2}$ if ${5}/{2} \le n<3.$

In $\S 4$ and $\S 5$, we investigate  conditions for the finite speed of propagation property (FSP)
 for  (1.1). Two different techniques have been developed for studying the  FSP property which both rely on energy and
 entropy estimates. The first method was developed by F. Bernis \cite{B1,B2}
 for the standard thin-film equation, (\ref{thinfilm}).
  For  $0<n<2,$  Bernis in \cite{B1} introduced the weighted  integral entropy function
  \begin{equation} \label{Bernis1}
  E_T(r):=\int_0^T \int_{-r}^r (r-|x|)^4 \Bigl(u^{\frac{\alpha+n+1}{{2}}}\Bigr)_{xx}^2 \, dx\,dt, \quad 0\le r<r_0, \quad 0<T,
  \end{equation}
  for $0<r_0<a,$ which corresponds to the third term on the left hand side of (\ref{localentropyintro}) with $\zeta(x)=(r- |x|)^4_+.$
  By taking $\hbox{\,supp\,} u_0 \subset\Omega\setminus \{ |x|<r_0\}$ and   deriving a nonlinear ordinary fourth order differential inequality
  for $E_T(r),$  it was possible to demonstrate the FSP property. For  $2 \le n <3,$ Bernis demonstrated the FSP property in \cite{B2}
  based on  a similar ordinary differential
  inequality for an appropriately defined weighted  integral energy function.

  An alternative method for studying  propagation properties for various thin-film like equations was proposed for the case
  $0<n<2$ in \cite{KSh} (see also the references and comments in \cite{HSh}). This method relies on functional rather than on differential
  inequalities for various entropy and energy functions. The emphasis in this technique is on cut off functions rather than on
  weight functions, which lead to functional dependence on parameterized subdomains of $\Omega.$ The main feature of the inequalities
  which are derived is that the minimal power of $u(x,t)$ in the terms on the left hand side of the inequality is strictly less
  than the minimal power of $u(x,t)$ which appears in the terms on the right hand side. Such inequalities in conjunction with the functional
  Stampacchia Lemma (see Lemma {6.1}) and its generalization for the system of functional inequalities ({see} Lemma {6.2}) can be used to demonstrate the FSP property in various contexts. In the present paper, we shall
  rely primarily on this second method.

  This latter method (the FI-method) has been used \cite{HSh} to obtain sharp
  estimates on the speed of propagation of the support of solutions to equation (1.1) with arbitrary $u_0\in L_1 (\Omega)$, $\Omega\setminus\text{supp}\,u_0 \neq \emptyset$.
   In \cite{DGS01}, using this method for  the thin film equation with forward diffusion (\ref{forwardsintro})
  with source type initial data,  $m_{crit}=n+2$  was shown to be a
   transitional value between the long and short time asymptotics of
   the Barenblatt porous media behavior and the predicted behavior of source type solutions for the standard thin film equation \cite{BPW}.
     Using the FI-method,    sharp sufficient conditions were obtained in \cite{DGGw} relating  the flatness at the edge of the
  support of $u_0$ to the phenomenon of a waiting time for propagation of the support of  solutions.
   Moreover  for equation (1.1) with an
  additional nonlinear absorption term, in \cite{Sh2} a  sharp
  sufficient condition on the flatness of $u_0$   guaranteeing
  the onset of shrinkage of the support was obtained, as well as estimates from below on the shrinkage rate.
    In future work, we hope to undertake similar analyzes for (\ref{G}) with $\nu=1.$

 For (\ref{G}) with "forward diffusion" (with  $ A=0$),
conditions for FSP  as well as sharp estimates for the  speed of propagation
were obtained in \cite{DGS01} for the  "strong slippage" case
$(0<n<2)$ only. The results which we obtain here are for (\ref{G}) with $\nu=-1,$ with $2<n<3,$ as well as for the much more delicate case of backward
diffusion, $\nu=1$, where the lower order diffusion term "encourages" the destruction of the
FSP property at all values of  $m$. Our analysis makes
use of some ideas from \cite{GS}. The proof given in $\S 4$ is for
the "strong slippage case" in which $0<n<2,$ and requires that
$m>n/2$ if $\nu=1.$ It relies on  the local entropy estimate
from $\S 3$ for $\alpha$ positive and
 the Stampacchia Lemma for systems. We demonstrate
that if $\hbox{\,supp\,} u_0 \subset \{ x \le 0\},$ then there
exists a continuous function, $s(t)$ satisfying $s(0)=0,$ and a
positive time $T_0$ such that $\hbox{\,supp\,} u(\cdot, \, t)
\subset [-a, \, s(t)],$ $s(t)<a \quad \forall\,t<T_0,$ and $s(T_0)=a.$
The proof in $\S 5$ is for the "weak slippage case" in which
$2<n<3,$ and requires that $m>n/2.$ It is based on  combining
 local entropy estimates for $-1<\alpha<0$ with the local
energy estimates from $\S 3$, and again makes use of the
Stampacchia Lemma for systems. We conjecture that the restriction
$m > n/2$ for $\nu=1$   is sharp or close to sharp. {Note
that in the stable case $\nu=-1,$ the value $m=\frac{n}{3}$ is critical
in the context of  asymptotics near  the edge of traveling wave solutions
with constant speed of propagation (see \cite{BP94}). Since the asymptotics in \cite{BP94} are valid also for the unstable
case $\nu=1$, we expect that FSP-property fails in the unstable case for some $m$ near
to $\frac{n}{3}.$ Thus  we suspect expect that for $\nu=1$, the restriction $m>n/2$ is sharp or close to sharp.

\smallskip
  As to further questions and future directions, a question of interest in the present context is to understand exactly how bad  the backwards diffusion can become while maintaining something of the smoothing properties of the
  thin film equation with $n>0$. Thus one should like to identify  transitional values in terms of rupture, positivity, touchdown, and infinite as opposed
  to finite spreading rates. Finally, one should like to extend all aspects of the analysis to
  higher dimensions.

\smallskip
The outline of  the  paper is as follows. The  existence of weak
non-negative solutions is proven in $\S 2$. Existence of strong
energy-entropy solutions is demonstrated in $\S 3.$ Finite speed of
propagation is proven in $\S 4$ for the case of weak slippage and in
$\S 5$ for the case of strong slippage.

\section{Weak solutions}
\setcounter{equation}{0}

In this section, we follow  Bernis \& Friedman \cite{BF90},
relying on  local parabolic regularity theory \cite{Eidelman69,G99}
to obtain global existence of weak solutions,  defined below. We regularize the initial data and use Gronwall's inequality in the
context of the entropy estimate. This allows us to avoid regularizing the lower order terms \cite{BP98,BP3}, and allows
us to widen the range of validity of the results.

\smallskip\par\noindent{\bf{Notation.}} Let
 $\Omega=(-a, \, a)$ where $a\in(0,\,\infty)$ is arbitrary,  $Q_t=\Omega\times(0, \, t),$
 $0<t<\infty,$
and set $P_t=\overline{Q}_t \setminus \{ u=0 \hbox{\, or \,} t=0\},$
$Q=Q_{\infty},$ $P=P_{\infty}.$

\bigskip
Let us  consider the problem
$$({\mathbb{P}})\left\{\begin{array}{l}
u_t+ (u^n(u_{xxx} +\nu u^{m-n} u_x - A u^{M-n} u_x))_x=0, \quad
(x, \,
t)\in Q_T,\\[1ex]
u_x(\pm a, \, t)= u_{xxx}(\pm a, \,
t)=0
\hbox{\, when \,} u(\pm a, \, t)\ne0, \quad t \in (0, \, T),\\[1ex]
 u(x, \,
0)=u_0(x), \quad x \in \overline{\Omega},
\end{array} \right.
$$
where $\nu= \pm 1$ and $0<T<\infty.$
We  shall assume  the initial conditions  to satisfy
\begin{equation} \label{ic} u_0\in H^1(\Omega),\quad u_0 \ge
0,\quad  u_0 \equiv\!\!\!\!\!\setminus\,\,\,\, 0.\end{equation}
 While we shall  look for solutions on a finite
interval, we can always consider the parallel Cauchy and periodic problems
obtained by extending the initial conditions via periodicity and
reflection. This will allow us, for example,  to directly
implement the generalized Bernis inequalities for nonnegative
periodic functions, \cite[Lemma B.1]{GS} in obtaining local energy estimates in $\S 3$. See also the remarks in \cite{BBD95,BP3}.

\smallskip
\begin{Def} \label{weaksolution} A function $u \in {\mathcal{C}}^{0, \, 1/2,
\, 1/8}({\bar{\Omega}} \times [0, \, \infty)) \cap L^{\infty}([0, \,
\infty);\, H^1(\Omega))$ is said to be a {\it{weak solution}} of
$({\mathbb{P}})$  if:
\smallskip\par\noindent(a)  $u \in C^{4,\,1}(P), \quad u \ge 0,$

\smallskip\par\noindent(b) $u_x(x, t)=u_{xxx}(x,t)=0$ when $u(x,t)
\ne 0,$ for $(x,t) \in \partial \Omega \times (0, \infty),$

\smallskip\par\noindent(c) $J\equiv u^{n}(u_{xxx} +\nu u^{m-n} u_x - A
u^{M-n} u_x) \in L^2(P),$

\smallskip\par\noindent(d) for all $\phi \in {\rm{Lip}}(\bar{\Omega} \times
(0, \, \infty))$ with compact support, $u$ satisfies:
\begin{equation} \label{weak}
\int_Q u \,\phi_t \, dx dt + \int_P u^n\,(u_{xxx} +\nu u^{m-n} u_x -
A u^{M-n} u_x) \,\phi_x \, dx dt =0,
\end{equation}

\smallskip\par\noindent(e) $u(x, \, 0)=u_0(x)$ for $x \in {\bar{\Omega}}$.
\end{Def}

\bigskip
Given this definition, we   formulate

\begin{Theorem} \label{weaksolutions} For the following range of parameter values:
$$\begin{array}{ll}
(i) & \nu=-1,\, 0<n,\, 0 \le A,\, n-2< m <M,\\[1ex]
(ii) & \nu=1,\, 0<n,\, 0<A,\, n-2< m <M,\\[1ex]
(iii) & \nu=1,\, 0<n,\, 0=A,\, n-2 < m < n+2,
\end{array}$$
there exists a solution to $(\mathbb{P})$ in the sense of Definition
\ref{weaksolution} for  $u_0$ satisfying
(\ref{ic}).
\end{Theorem}

We shall find a weak solution to ($\mathbb{P}$) as
the limit of a subsequence of smooth positive solutions to a
regularized  problem, with
regularized initial conditions, ${u_0}_{\epsilon}.$ We shall require
that for  $\epsilon>0$ and for  some $\lambda \in (0, \, 1),$ $\theta \in (0, \, 2/5],$
${u_0}_{\epsilon}$ satisfies
\begin{equation} \label{icr}
\begin{array}{l}
 {u_0}_{\epsilon} \in {\mathcal{C}}^{4, \, \lambda}(\bar{\Omega}),
 \quad
   {u^{\prime}_{\epsilon}}_0(\pm a)=u^{\prime\prime\prime}_{{\epsilon}_0}(\pm
 a)=0, \quad u_0 + \epsilon^{\theta} \le {u_0}_{\epsilon} \le u_0
 +1,\\ [1ex]
  {u_{\epsilon}}_0 \rightarrow u_0 \hbox{\, in \,} H^1((-a, \, a))
 \hbox{\, as \,} \epsilon \rightarrow 0.
 \end{array}
 \end{equation}
Following  \cite{BF90}, we set
\begin{equation} \label{deff}
f_{\epsilon}(s)=\frac{|s|^{n+4}}{\epsilon |s|^n + s^4};
\end{equation}
when $0<s \ll 1,$ $f_{\epsilon}(s) \approx \frac{s^4}{\epsilon}$ if $0<n<4$ and $f_{\epsilon}(s) \approx|s|^n$ if $n \ge 4,$
which shall allow us to guarantee positivity of the approximants $u_{\epsilon}$ for $t>0$.
  Here and in the section which follows,
$c_i,$ $d_i$  denote  positive constants that are independent of
$\epsilon,$ and $C_i(t)$ denotes a positive increasing function
defined on $(0, \, \infty)$ that is independent of $\epsilon;$
$c_i,$ $d_i,$  $C_i(t)$ may depend on $\Omega,$ $u_0,$ and the
problem parameters, and their value may change from line to line.

\begin{proof}
Let us define the approximating  problem
$({\mathbb{P}}_{\epsilon})$
$$({\mathbb{P}}_{\epsilon}) \left\{\begin{array}{ll}
u_t + \{f_{\epsilon}(u)(u_{xxx} + \nu u^{m-n} u_x - A u^{M-n} u_x
) \}_x=0,& (x, \, t) \in Q_T,\\[1ex]
u_x(\pm a,\,t)=u_{xxx}(\pm a,\,t)=0,& t \in (0, \, T),\\[1ex]
 u(x, \,
0)={u_0}_{\epsilon}(x), & x \in \overline{\Omega}.
\end{array} \right.
$$
Problem $({\mathbb{P}}_{\epsilon})$ possess a unique maximal
positive solution, $u_{\epsilon},$ such that $u_{\epsilon} \in
{\mathcal{C}}^{4,\, \lambda,\, \lambda/4\,}(\bar{\Omega} \times
[0, \, \tau_{\epsilon})),$ $\tau_{\epsilon} > 0$, see
\cite[Theorem 6.3, p 302]{Eidelman69} as well as the remark
following the proof given there. That the associated Cauchy problem with periodically reflected
initial data
maintains the periodicity and reflection properties,
can be seen by translating and reflecting the
solution, then  invoking uniqueness of {the} solutions to the Cauchy
problem.

 Testing
$({\mathbb{P}}_{\epsilon})$ with $\phi \equiv 1$ and recalling
(\ref{ic}), (\ref{icr}), it follows that
\begin{equation} \label{meanconserv}
0 < {\overline{u_{\epsilon}}}(t) = \overline{{u_0}_{\epsilon}} \le
\overline{u}_0 +1, \quad t \in (0, \, T),
\end{equation}
where $\overline{v}:=|\Omega|^{-1}\int_{\Omega} v.$ The equality ${\overline{u_{\epsilon}}}(t) = \overline{{u_0}_{\epsilon}}$ in (\ref{meanconserv})
expresses {\it{mass conservation}}. Note that
(\ref{meanconserv})  also  holds  for  solutions to similarly defined approximating
problems for  (\ref{gthermoc2}),
(\ref{attractive}\,i).

 Setting
$$h(s)=\left\{\begin{array}{ll}
\frac{\nu s^{m-n+1}}{m-n+1} - A \frac{s^{M-n+1}}{M-n+1}, & m,M
\ne \{n-1\},\\[1ex]
\nu \ln s - A \frac{s^{M-n+1}}{M-n+1},& m=n-1,\\[1ex]
\frac{\nu s^{m-n+1}}{m-n+1} - A \ln s, & M=n-1,
\end{array}\right.$$
and testing $({\mathbb{P}}_{\epsilon})$ with
$-{u_{\epsilon}}_{xx}-h(u_{\epsilon}),$ we obtain the \textit{energy estimate} \cite{BF90,EG,Derjaguin}
\begin{multline}\label{2.5'}
\int_{\Omega} \Bigl[\frac{1}{2}{u_{\epsilon}}_x^2 -
H({u_{\epsilon}})\Bigr]\, dx + \int_{Q_t}
f_{\epsilon}(u_{\epsilon})({u_{\epsilon}}_{xxx} +
h'(u_{\epsilon}){u_{\epsilon}}_x)^2\, dx\, dt=\hspace{1cm}\\  \int_{\Omega}
\Bigl[ \frac{1}{2}{{u_0}_{\epsilon}}_x^2 -
H({u_0}_{\epsilon})\Bigr]\, dx,\hspace{2cm}
\end{multline}
where
$$H(s)=\left\{\begin{array}{ll}
\frac{\nu s^{m-n+2}}{(m-n+2)(m-n+1)} - \frac{A
s^{M-n+2}}{(M-n+2)(M-n+1)}, & m,M
\ne \{n-1\},\\[1ex]
\nu (s \ln s -s) - \frac{A s^{M-n+2}}{(M-n+2)(M-n+1)},& m=n-1,\\[1ex]
\frac{\nu s^{m-n+2}}{(m-n+2)(m-n+1)} - A (s\ln s-s), & M=n-1.
\end{array}\right.$$

\bigskip
Let $\nu=1$ and  $-1 \le m-n <2.$ Then noting that
\begin{equation} \label{a00}
-1 \le s \ln s -s < s^2, \quad 0 <s,
\end{equation}
recalling the Gagliardo-Nirenberg inequality
$$||u||_{L^p(\Omega)} \le c_1 ||u_x||_{L^2(\Omega)}^{1/2} ||u||_{L^1(\Omega)}^{1/2} + c_2
||u||_{L^1(\Omega)}, \quad 1 <p<4,
$$
and that the approximating solutions are  positive and satisfy
(\ref{meanconserv}), it follows that
\begin{equation} \label{A1}
\int_{\Omega} H(u_{\epsilon})\, dx \le \frac{1}{4} \int_{\Omega}
{u_{\epsilon}}_x^2 \,dx + c_3 \Bigl(\int_{\Omega}
u_{\epsilon} \, dx \Bigr)^{c_4}=\frac{1}{4} \int_{\Omega}{u_{\epsilon}}_x^2\, dx
+ c_5.
\end{equation}
Suppose  the restrictions on the parameters stated in Theorem 1 hold
and that $-2<m-n<-1$ if $\nu=1.$
 Then recalling (\ref{a00}), and noting that
for $A_1, \, A_2>0,$ $0< \alpha < \beta,$
$$A_1 s^{\alpha} - A_2 s^{\beta} <  A_1 \Bigl( \frac{A_1 \alpha}{A_2 \beta}\Bigr)^{\frac{\alpha}{\beta-\alpha}},
\quad 0<s,$$
and that by H\"{o}lder's inequality, if $0< \alpha <1,$ then
$$\int_{\Omega} |u|^{\alpha} \, dx \le \alpha \int_{\Omega} |u| \, dx + (1-\alpha)|\Omega|,$$
we obtain that
\begin{equation} \label{A2}
\int_{\Omega} H(u_{\epsilon})\,dx \le c_6 (\overline{u}_0 +1)=c_7.
\end{equation}

From (\ref{2.5'}), (\ref{A1}), (\ref{A2}), we may now conclude  that
\begin{equation} \label{energyestimate}
\frac{1}{4} \int_{\Omega} {u_{\epsilon}}_x^2(x,\,t)\, dx + \int_{Q_t}
f_{\epsilon}(u_{\epsilon})( {u_{\epsilon}}_{xxx} + \nu
{u_{\epsilon}}^{m-n} {u_{\epsilon}}_x - A {u_{\epsilon}}^{M-n}
{u_{\epsilon}}_x)^2\, dx\, dt \le c_8.
\end{equation}

\begin{Remark}
The constraint that $m-n<2$ when $\nu=1,$  $A=0$ has been imposed to guarantee global existence.
If $\nu=1,$  $A=0,$
and $m-n=2,$   (\ref{energyestimate}) remains valid
if $\overline{u_0}$ is sufficiently small, and global existence is again implied. See \cite{BP98}.
\end{Remark}

\smallskip
From (\ref{meanconserv}), (\ref{energyestimate}), we obtain that
\begin{equation} \label{a0}
||u_{\epsilon}||_{L^{\infty}(0, \, t; \, H^1(\Omega))} \le c_8, \quad  ||f_{\epsilon}^{1/2}({u_{\epsilon}}_{xx} + h(u_{\epsilon}))_x\,||_{L^2(Q_t)} \le c_9.
\end{equation}
 The following argument from \cite{BF90} is by now standard. Noting that $u_t=-J_x,$
   (\ref{a0}) implies that
\begin{equation} \label{a0a}
||u_t||_{L^2(0, \, t; \, H^{-1}(\Omega))}, \; ||J||_{L^2(Q_t)} \le
c_{10},
\end{equation}
and the estimates (\ref{a0}), (\ref{a0a}) can be seen to imply the uniform H\"{o}lder estimate
\begin{equation} \label{Holder}
||u_{\epsilon}||_{{\mathcal{C}}^{0, \, 1/2, \, 1/8}(\overline{Q_t})}
\le c_{11}.
\end{equation}
See \cite{BF90} for details.

\begin{Remark} It is also possible to implement the above
discussion when $u^n$ is replaced by $f(u)$ in ($\mathbb{P}$), for
 $f(u) \in {\mathcal{C}}(0, \, \infty) \rightarrow
\mathbb{R}^+,$ and to work with \cite{BF90}
$$f_{\epsilon}(u)=\frac{f(u) u^4}{\epsilon f(u) + u^4},$$
under suitable assumptions on $f(u)$.\end{Remark}

\bigskip
We now demonstrate roughly as in \cite{BF90,G99}
 that $u_{\epsilon} \ge 4 \sigma
>0$ in ${\overline{\Omega}}\times[0, \, \tau_{\epsilon}],$ where $\sigma=\sigma(\epsilon).$
On $[0, \, \tau_{\epsilon}),$ recalling (\ref{Holder}) we know that $0< u_{\epsilon}(x, \, t)
<\tilde{A},$ for some $\tilde{A}$ which is independent of $\epsilon$. Multiplying the  equation  in
$({\mathbb{P}}_{\epsilon})$ by $G_{\epsilon}'(u_{\epsilon})$ \cite{NCQ,BF90}, where
$$G_{\epsilon}(s)=-\int_s^{\tilde{A}} g_{\epsilon}(r) \, dr, \quad
g_{\epsilon}(s)=-\int_s^{\tilde{A}} \frac{dr}{f_{\epsilon}(r)},$$
 and integrating,
\begin{equation}\label{47}
\int_{\Omega} G_{\epsilon}(u_{\epsilon}(x,t))\, dx + \int_{Q_t}
[{u_{\epsilon}}_{xx}^2 - ( \nu {u_{\epsilon}}^{m-n} -  A
{u_{\epsilon}}^{M-n}) {u_{\epsilon}}^2_x]\, dx\, dt=\int_{\Omega}
G_{\epsilon}({u_0}_{\epsilon})\, dx.
\end{equation}
If $m-n \neq -1,$ then  integrating the term
$-\int_{Q_T} \nu u_{\epsilon}^{m-n} {u_{\epsilon}}_x^2$ by parts,
\begin{multline} \label{12}
\int_{\Omega} G_{\epsilon}(u_{\epsilon}(x,\, t))\, dx + \int_{Q_t}
{u_{\epsilon}}_{xx}^2\, dx\, dt +
A \int_{Q_t} {u_{\epsilon}}^{M-n} {u_{\epsilon}}_x^2 \, dx\, dt\\
\le \frac{1}{2} \int_{Q_t} {u_{\epsilon}}_{xx}^2 \, dx dt + \frac{1}{2(m-n +
1)^2} \int_{Q_t} {u_{\epsilon}}^{2 m- 2n +2}\, dx\, dt + \int_{\Omega}
G_{\epsilon}({u_0}_{\epsilon})\, dx,
\end{multline}
 recalling (\ref{a0})
 \begin{equation} \label{13}
\int_{\Omega} G_{\epsilon}(u_{\epsilon}(x,\,t))\,dx + \int_{Q_t}
\frac{1}{2}{u_{\epsilon}}_{xx}^2 \,dx\,dt  \le c_{12} \int_{Q_t}
{u_{\epsilon}}^{-2}\,dx\,dt   + \int_{\Omega}
G_{\epsilon}({u_0}_{\epsilon})\,dx.
\end{equation}
If $m-n=-1$, the term $(\ln u_{\epsilon})^2$ replaces
$\frac{{u_{\epsilon}^{2 m -2n  +2}}}{(m-n  + 1)^2}$ in (\ref{12}).
Then noting that
$$ \ln^2(s) \le {c}_{13} s^{-2}, \quad  0 < s< \tilde{A}<\infty,$$
the estimate (\ref{13}) again follows.

Since by (\ref{deff}),
\begin{equation} \label{2.17m}
\epsilon s^{-4} \le \frac{1}{f_{\epsilon}(s)},\quad 0 < s,
\end{equation}
it now follows easily  that
\begin{equation} \nonumber
\int_{\Omega} G_{\epsilon}({u_{\epsilon}}(x,t))\, dx + \int_{Q_t}
\frac{1}{2}{u_{\epsilon}}_{xx}^2\,dx\,dt  \le \frac{c_{14}}{\epsilon}
\int_{Q_t} G_{\epsilon}(u_{\epsilon}) \,dx\, dt +  c_{15} t + \int_{\Omega}
G_{\epsilon}({u_0}_{\epsilon})\, dx.\end{equation}
Using (\ref{icr}), (\ref{2.17m}), we find that $\int_{\Omega} G_{\epsilon}({u_0}_{\epsilon})\, dx \le c_{16}.$ Now we may use Gronwall's
inequality to conclude that
\begin{equation} \label{49}
\int_{\Omega} G_{\epsilon}(u_{\epsilon}(x,\,t))\,dx \le D_{\epsilon}(t) <
\infty, \quad t \in [0, \, \tau_{\epsilon}).
\end{equation}
where for all $0<\epsilon\ll 1,$ $D_{\epsilon}(t)$ is a positive
increasing function of $t$ defined on $[0, \, \infty).$ As in
\cite{BF90}, (\ref{49}) can be seen to imply positivity.

\bigskip
The solution, $u_{\epsilon}(x, \, t),$  may now be extended to exist
globally, as in \cite{BF90,G99}. Select $\tilde{f}_{\epsilon}(s) \in
{\mathcal{C}}^2 (\mathbb{R})$ such that $\tilde{f}_{\epsilon}(s)
\equiv f_{\epsilon}(s)$ for $s \ge 2 \sigma,$ and
$\tilde{f}_{\epsilon}(s) \ge f_{\epsilon}(\sigma)$ for all $s \in
\mathbb{R}.$ Thus  $u_{\epsilon}(x, \, t)$ also constitutes a weak
solution of
$${u_{\epsilon}}_t +
\{\tilde{f}_{\epsilon}(u_{\epsilon})({u_{\epsilon}}_{xxx} +
h'({u_{\epsilon}}) {u_{\epsilon}}_x)\}_x =0,$$ satisfying the same
initial and boundary conditions as before. For $x_{\epsilon} \in
[-a, \, a),$ set
$$v_{\epsilon}(x, \, t)=\int_{x_{\epsilon}}^x u_{\epsilon}(\xi, \,
t) \, d\xi - \int_0^t
\tilde{f}_{\epsilon}(u_{\epsilon}(x_{\epsilon}, \, \theta))
\{{u_{\epsilon}}_{xxx} + h'({u_{\epsilon}}) {u_{\epsilon}}_x
\}(x_{\epsilon}, \, \theta) \, d\theta.$$ The regularity and
positivity of ${u_{\epsilon}}(x, \, t)$ imply that $v_{\epsilon}(x,
\, t)$ is well defined in $D=\bar{\Omega} \times (0, \,
\tau_{\epsilon})$ and satisfies
\begin{equation} \label{50}
\left\{\begin{array}l {v_{\epsilon}}_t +
\tilde{f}_{\epsilon}(u_{\epsilon}(x, \, t)) \{ {v_{\epsilon}}_{xxxx}
+ \nu{{u_{\epsilon}}}^{m-n} {v_{\epsilon}}_{xx} - A
{{u_{\epsilon}}}^{M-n} {v_{\epsilon}}_{xx} \}
=0,\\[1ex]
v_{\epsilon}(\pm\, a,t)={v_{\epsilon}}_{xx}(\pm\,
a,t)=0.\end{array}\right.
\end{equation}
Using parabolic regularity results for $v_{\epsilon},$  enhanced
regularity may be obtained for $u_{\epsilon}$ and hence for the
$u_{\epsilon}$-dependent coefficients in (\ref{50}). Returning again
to (\ref{50}),  additional regularity is obtained for
$v_{\epsilon},$ which  allows us to conclude that ${u_{\epsilon}}(x,\tau_{\epsilon})
\in {\mathcal{C}}^{4,\, \lambda}(\bar{\Omega}).$ Therefore the solution may be continued, in
contradiction to the assumed maximality of the solution.

\smallskip
Having demonstrated the global existence of positive approximants, $u_\epsilon,$
  existence of a sequence
$\{u_{{\epsilon}_k}\}$ converging uniformly to a solution of
$({\mathbb{P}})$ on ${\overline{Q}}_T$, for all $0 < T < \infty,$ as
$\epsilon_k \rightarrow 0,$ is now implied by Arz\'{e}la-Ascoli and the uniform H\"{o}lder estimates, (\ref{Holder}), as in \cite{BF90}.
 \end{proof}

{
\begin{Remark} It can be readily verified that the results of Theorem \ref{weaksolutions} remain valid when
$m-n=-2$ and $\nu=-1$, if the initial conditions satisfy the additional constraint that $\int_{\Omega} H(u_0) \, dx < \infty.$
\end{Remark}}

\begin{Remark}
 The existence of weak solutions for
 (\ref{gthermoc2}), (\ref{attractive}\,i), with $h'(u)$
replacing $\nu u^{m-n}  -A u^{M-n}$
 in Definition \ref{weaksolutions},
 can  be
concluded
  for  initial data satisfying
(\ref{ic}), (\ref{icr}),  by verifying that
(\ref{A1})
 holds for (\ref{gthermoc2}), that (\ref{A2}) holds for
(\ref{attractive}\,i), and  that (\ref{13}) holds
 for both (\ref{gthermoc2}) and (\ref{attractive}\,i), then arguing as above.
\end{Remark}

\bigskip
\section{Strong entropy-energy solutions}
\setcounter{equation}{0}

To get strong entropy-energy solutions, we  derive local entropy estimates
 \cite{BBD95}, which give us strong solutions \cite{BBD95,CPAM96}, then derive a local energy estimate \cite{GS}.
 In obtaining the local entropy estimates and strong solutions, we follow \cite{BBD95} closely.
Throughout this section the parameters will be assumed to satisfy the conditions in Theorem 1.
 Moreover, in referring to
solutions of  $({\mathbb{P}})$ and $({\mathbb{P}}_{\epsilon})$, we
shall assume that $u_0$ satisfies (\ref{ic}) and ${u_0}_{\epsilon}$
satisfies (\ref{icr}). Some further restrictions shall be introduced
in the sequel.

\smallskip
Before deriving the entropy estimates, we present a lemma, which constitutes a refinement of Theorem 3.1 in \cite{BBD95},
and which is
useful for concluding regularity results from entropy estimates.

\begin{Lemma} \label{regularity}
Let $u(x,t)$ be a weak solution of $({\mathbb{P}})$ obtained as the
limit of a subsequence of  solutions $u_{\epsilon}(x, \, t)$ of
$({\mathbb{P}}_{\epsilon}).$   Suppose that  for some
$\alpha \in (\frac{1}{2} -n, \, 2-n),$ there exist constants $c_1,$
$c_2,$ and $\delta>0,$ which do not depend on $\epsilon$, such that
\begin{equation} \label{a2.34}
\int_{Q_T} u_{\epsilon}^{\alpha+n-2 \gamma
+1}(u_\epsilon^\gamma)_{xx}^2\,dx\,dt  \le c_1,
\end{equation}
and
\begin{equation} \label{a2.35}
\int_{Q_T} u_{\epsilon}^{\alpha + n -3} {u_{\epsilon}}_x^4 \,dx\,dt
 \le c_2,
\end{equation}
for all $\gamma$ satisfying
\begin{equation}\label{a.47}
 \frac{1 + n + \alpha}{3} \le \gamma \le \frac{1 + n + \alpha}{3} + \delta,
 \end{equation}
then $u^{1/\beta}(\cdot, \, t) \in {\mathcal{C}}^1([-a, \, a])$ for
all $\beta \in (0, \, \frac{3}{n+\alpha+1})$ for almost every $t>0.$
\end{Lemma}

\begin{proof}
For any $0< \beta< \frac{3}{n+ \alpha +1},$ we may choose $\gamma$
satisfying (\ref{a.47}) such that  $0< \beta \gamma <1.$ Setting
$q=4 - \frac{(1+n+\alpha)}{\gamma}$ and arguing  as in the proof of
\cite[Lemma 3.1]{BBD95},  it follows from
(\ref{a2.34}),(\ref{a2.35})  that  for almost every $t>0$ there
exists a $C_1(t)<\infty$ such that
\begin{multline} \label{a3.8}
\hbox{\, if  \,} u(y,\, t)=0 \hbox{\, for some \,} y \in [-a,\, a],
\hbox{\, then \,} \\
 |(u^{\gamma})_x|^{(4-q)/q}(x,\, t)
\le C_1(t)|x-y|^{(q-1)/q}\quad \hbox{\, for \,} x \in [-a, a].
\end{multline}
From (\ref{a3.8}), we find by integrating that for almost every
$t>0,$ there exists a $C_2(t)<\infty$ such that
\begin{multline} \label{a3.5}
 \hbox{\,if \,} u(y,\,t)=0
\hbox{\, for some \,} y \in [-a,\, a], \hbox{\, then \,}\\
 u(x,t) \le C_2(t)
|x-y|^{\frac{3}{\alpha + n +1}} \quad \hbox{\, for \,} x \in [-a,
a].
\end{multline}

Since $0< \beta \gamma <1,$ we may combine (\ref{a3.5}) and
(\ref{a3.8})  to obtain that for almost every $t>0,$ there exists a
$C_3(t) < \infty$ such that
\begin{multline} \label{a3.9}
\hbox{\, if \,} u(y,\, t)=0 \hbox{\, for some \,} y \in [-a,\, a],
\hbox{\, then for \,} x \in [-a, a],\\  |(u^{1/\beta})_x(x, \, t) |
\le C_3(t)|x-y|^{\frac{3}{4-q}}|x-y|^{\frac{q-1}{4-q}} \le C_3(t)
|x-y|^{\mu},
\end{multline}
where $\mu=\frac{1}{\beta}\frac{3}{\alpha+n+1}-1>0$ and
$C_3(t)<\infty.$
\end{proof}

From Lemma \ref{regularity}, two simple but useful corollaries
follow.

\begin{Corollary} \cite{BBD95} \label{2strong}
Let $u(x,t)$ be a weak solution of $({\mathbb{P}})$ obtained as the
limit of a subsequence of  solutions $u_{\epsilon}(x, \, t)$ of
$({\mathbb{P}}_{\epsilon}).$   Suppose that  for some
$\alpha \in (\frac{1}{2} -n, \, 2-n),$ there exist constants $c_1,$
$c_2,$ and $\delta>0,$ which do not depend on $\epsilon$, such that
for all $\gamma$ satisfying (\ref{a.47}), the estimates
(\ref{a2.34}) and (\ref{a2.35}) hold. Then $u(\cdot, \, t) \in
{\mathcal{C}}^1([-a,a])$ for almost every $t>0.$
\end{Corollary}

\begin{proof}
If $\alpha \in (\frac{1}{2}-n, \, 2-n)$, then
$\frac{3}{\alpha+n+1} \in (1,\,2).$ Hence $1 \in (0, \,
\frac{3}{n+\alpha+1}).$
\end{proof}

\begin{Remark}\label{2strong'}
If $u(\cdot, \, t) \in {\mathcal{C}}^1([-a,a])$ for almost
every $t>0,$  then $u(x,\,t)$ is said to be a  {\it{strong solution}}
 in the sense of Bernis \cite{B1}.
\end{Remark}

\begin{Corollary}\label{3.4}
Let $u(x,t)$ be a weak solution of $({\mathbb{P}})$ obtained as the
limit of a subsequence of  solutions $u_{\epsilon}(x, \, t)$ of
$({\mathbb{P}}_{\epsilon}).$  Let  $\Psi$ denote a nonempty
subset of $(\frac{1}{2} -n, \, 2-n).$ If for all $\alpha \in \Psi,$
 there exist constants $c_1,$ $c_2,$ and $\delta>0,$  which do not depend on $\epsilon$,
 such that for all $\gamma$ satisfying (\ref{a.47}), the estimates
(\ref{a2.34}) and (\ref{a2.35}) hold, then $u^{1/\beta}(\cdot, \, t)
\in {\mathcal{C}}^1([-a, \, a])$ for all $\beta \in (0, \,
\frac{3}{n+ \inf \Psi +1}\,)$ for almost every $t>0.$
\end{Corollary}

\begin{proof} The result is an immediate consequence of Lemma
\ref{regularity}. \end{proof}

\smallskip\par\noindent We now derive our  primary entropy estimates. \smallskip

Let
\begin{equation} \label{Azeta}
\zeta \in {\mathcal{C}}^4([-a,\,a]) \hbox{\, with support in \,}(-a,
\, a) \hbox{\, and \,} \zeta \ge 0,
\end{equation}
or $\zeta \equiv 1,$ and let \cite{BBD95,CPAM96}
$$G_{\epsilon}(s)=\frac{\epsilon s^{\alpha + n -3}}{(\alpha+n
-4)(\alpha + n-3)} + \frac{s^{\alpha+1}}{\alpha(\alpha+1)},$$ where
$\alpha \in (1/2-n, \, 2-n) \setminus \{0, \, -1\}.$ Using $\zeta^4
G_{\epsilon}'(u_{\epsilon})$ to test $({\mathbb{P}}_{\epsilon})$ on
$Q_T=\Omega \times (0, \, T),$ $0<T<\infty,$ and treating the terms
which also appear in the classical thin film equation as they were
treated in \cite{BBD95}, we obtain that for any $\gamma$
satisfying
\begin{equation} \label{a2.6}
\frac{t+1 - \sqrt{(t-2)(1-2t)}}{3} < \gamma < \frac{t+1 +
\sqrt{(t-2)(1-2t)}}{3},
\end{equation}
where $t=\alpha +n,$ there exist  positive constants, $c_{3},$
$c_{4},$  which do not depend on $\epsilon,$ such that
\begin{multline} \label{entropy}
\int_{\Omega} \zeta^4 G_{\epsilon}(u_{\epsilon}(x, \, T)) \, dx
+\\
c_{3} \Biggl[ \int_{Q_T} \zeta^4 {u_{\epsilon}}^{\alpha + n -2
\gamma +1} (u_{\epsilon}^{\gamma})_{xx}^2\,dx\,dt + \int_{Q_T} \zeta^4
{u_{\epsilon}}^{\alpha +n-3} {u_{\epsilon}}_x^4 \, dx\,dt \Biggr] \le\\
\int_{\Omega} \zeta^4 G_{\epsilon}({u_0}_{\epsilon}) \, dx +
c_{4} \int_{Q_T}( |\zeta_x|^4+ |\zeta \zeta_{xx}|^2) u^{n+\alpha
+1}_\epsilon\,dx\,dt +I,
\end{multline}
where
$$I:=-\int_{Q_T} \zeta^4
g_{\epsilon}(u_{\epsilon})\{f_{\epsilon}(u_{\epsilon})(\nu
{u_{\epsilon}}^{m-n} {u_{\epsilon}}_x - A {u_{\epsilon}}^{M-n}
{u_{\epsilon}}_x)\}_x \,dx\,dt,\quad
g_{\epsilon}(u_{\epsilon}):=G'_{\epsilon}(u_{\epsilon}).
$$ Integrating $I$ by parts,
$$I= \int_{Q_T} \zeta^4 g_{\epsilon}'(u_{\epsilon})
f_{\epsilon}(u_{\epsilon})(\nu {u_{\epsilon}}^{m-n}
{u_{\epsilon}}_x^2 -A {u_{\epsilon}}^{M-n}
{u_{\epsilon}}_x^2)\,dx\,dt+\hspace{3cm}$$
$$\hspace{2cm}\int_{Q_T} 4 \zeta^3 \zeta_x g_{\epsilon}(u_{\epsilon})
f_{\epsilon}(u_{\epsilon}) (\nu {u_{\epsilon}}^{m-n}
{u_{\epsilon}}_x - A {u_{\epsilon}}^{M-n} {u_{\epsilon}}_x)\,dx\,dt := I_a
+ I_b.$$

 The term $I_a$ may be written as
 \begin{equation} \label{Iam}
I_a=\int_{Q_T} \zeta^4 (\nu {u_{\epsilon}}^{\alpha + m-1} - A
{u_{\epsilon}}^{\alpha + M-1}){u_{\epsilon}}_x^2\,dx\,dt.\end{equation} For
$\nu=-1,$  both terms in (\ref{Iam}) are non-positive. For
$\nu=+1$,
  we
estimate
\begin{multline} \label{Iap}
 I_a \le \delta \int_{Q^T} \zeta^4 {u_{\epsilon}}^{\alpha+n-3}
{u_{\epsilon}}_x^4 \,dx\,dt + c_{5}(\delta)\int_{Q_T} \zeta^4
{u_{\epsilon}}^{\alpha+ 2m -n +1} \,dx\,dt -\\A \int_{Q_T} \zeta^4
{u_{\epsilon}}^{\alpha + M -1}{u_{\epsilon}}_x^2\, dx \,dt,
\end{multline}
 where $\delta>0$ is
arbitrary.

With regard to $I_b$, integration by parts gives  that
$$I_b = - \int_{Q_T} 4 (\zeta^3 \zeta_x)_x
\Bigl[\int^{u_{\epsilon}}_{0} g_{\epsilon}(s) f_{\epsilon}(s) [\nu
s^{m-n} - A s^{M-n}] \, ds \Bigr]\,dx\,dt.$$ As noted in \cite{BBD95},
$|g_{\epsilon}(u_{\epsilon}) f_{\epsilon}(u_{\epsilon})|
\le c_{6} u_{\epsilon}^{n+\alpha}.$  Thus, recalling (\ref{a0}) and
that $M>m,$
\begin{equation} \label{Ibp}
I_b \le c_{7} \int_{Q_T} |(\zeta^3 \zeta_x )_x|
{u}_{\epsilon}^{\alpha + m + 1} \, dx \, dt.
\end{equation}

If $\nu=-1$, we may combine the estimates on $I_a$ and $I_b$ to
obtain
\begin{multline} \label{entropy-}
\int_{\Omega} \zeta^4 G_{\epsilon}(u_{\epsilon}(x, \, T)) \, dx +A\int_{Q_T}\zeta^4u_\epsilon^{\alpha+M-1}{u_\epsilon}_x^2\, dx\,dt+\\
c_{3} \Biggl[ \int_{Q_T} \zeta^4 {u_{\epsilon}}^{\alpha + n -2
\gamma+1} (u_{\epsilon}^{\gamma})_{xx}^2 \,dx\,dt+ \int_{Q_T} \zeta^4
{u_{\epsilon}}^{\alpha +n-3} {u_{\epsilon}}_x^4 \,dx\,dt
\Biggr]+\\\int_{Q_T}\zeta^4u_\epsilon^{\alpha+m-1}
{u_\epsilon}_x^2 \,dx\,dt \le \int_{\Omega} \zeta^4 G_{\epsilon}({u_0}_{\epsilon})
\, dx +II,
 \end{multline}
where
$$II=c_{4} \int_{Q_T}( |\zeta_x|^4+ |\zeta \zeta_{xx}|^2) u_{\epsilon}^{n+\alpha
+1} \,dx\,dt + c_{7} \int_{Q_T} |(\zeta^3 \zeta_x )_x| {u}_{\epsilon}^{\alpha
+ m + 1}\,dx\,dt.$$

Similarly, if $\nu=+1$, the  estimates yield
\begin{multline} \label{entropy+}
\int_{\Omega} \zeta^4 G_{\epsilon}(u_{\epsilon}(x, \, T)) \, dx +A
\int_{Q_T} \zeta^4 {u_{\epsilon}}^{\alpha + M
-1}{u_{\epsilon}}_x^2 \,dx\,dt+\\
c_{8} \Biggl[ \int_{Q_T} \zeta^4 {u_{\epsilon}}^{\alpha + n -2
\gamma +1} (u_{\epsilon}^{\gamma})_{xx}^2 \,dx\,dt + \int_{Q_T} \zeta^4
{u_{\epsilon}}^{\alpha +n-3} {u_{\epsilon}}_x^4 \,dx\,dt\Biggr]  \le
\\ + c_{5} \int_{Q_T} \zeta^4 {u_{\epsilon}}^{\alpha+ 2 m -n +1} \,dx\,dt+
\int_{\Omega} \zeta^4 G_{\epsilon}({u_0}_{\epsilon}) \,dx+
II.
\end{multline}

To obtain bounds from  (\ref{entropy-}), (\ref{entropy+}), we
impose certain conditions  on $\alpha$ and on the initial data.

\begin{Remark} \label{alphabar}
Suppose that $u_0$   satisfies (\ref{ic}) and $0<n<3$. Defining
\begin{equation} \label{s17m}
\alpha^{\ast} =\left\{
\begin{array}{ll} \frac{1}{2}-n, &\quad 0<n \le \frac{3}{2},\\[1ex]
-1, & \quad \frac{3}{2} < n <3, \end{array} \right. \end{equation}
we see that $\alpha^{\ast} + 1 \ge 0,$  $\alpha^{\ast} \in
[\frac{1}{2} -n, \, 2-n),$ and
$\int_{\Omega} \zeta^4 u_0^{\alpha+1} \, dx < +\infty$
for all $\alpha \in
({\alpha}^{\ast},\, 2-n)$ and $\zeta \in {\mathcal{C}}^4([-a,\,a]).$
\end{Remark}

In consideration of the above remark, we  define

\begin{Def} \label{alphazeta}
Suppose that $u_0$   satisfies (\ref{ic}) and $\zeta \in
{\mathcal{C}}^4([-a,\,a]).$ Then we define $\alpha_0(\zeta) \equiv
\inf {\alpha}$ such that $\alpha > \frac{1}{2}-n$ and
\begin{equation}\nonumber
\begin{array}{ll}
\int_{\Omega} \zeta^4 u_0^{\alpha+1} \, dx < +\infty & \hbox{\,
if \,} \alpha \ne -1,\\ [1ex] \int_{\Omega} \zeta^4 | \ln u_0| \,
dx < +\infty & \hbox{\, if \,} \alpha = -1.\end{array}
\end{equation}
\end{Def}
Remark \ref{alphabar} and the definition of
$\alpha_0(\zeta)$  imply that if $0<n<3,$ then
\begin{equation} \label{s17}
\frac{1}{2} -n \le \alpha_0(\zeta) \le \alpha^{\ast}<2-n.
\end{equation}

%

\bigskip
{With regard to the stable case,
the theorem below follows essentially as in
 \cite{BBD95,DGS01}.}

\begin{Theorem} \label{Case3a} {\bf{(The stable case.)}} Suppose that $\nu=-1,$  $0 \le
A,$ $0<n,$  {$-2 \le m-n,$ with $m < M$ if $A>0,$ and with $\int_{\Omega} u_0 \, dx < \infty$ if $-2=m-n$.}

\smallskip\par\noindent
$i)$ Let $\beta \in (0, \, \beta_0)$ where $\beta_0=\max \Big\{ \frac{3}{n+\alpha_0 +1}, \frac{1}{m+\alpha_0 +1} \Big\}$ and  $\alpha_0=\alpha_0(\zeta = 1).$  Then $u^{1/\beta}(\cdot, \, t) \in
{\mathcal{C}}^1([-a, \, a])$ for almost every $t>0.$

\smallskip\par\noindent
$ii)$ Let $\zeta$ satisfy (\ref{Azeta}).  Then, for any  $\alpha \in
(\max\{\alpha_0(\zeta), -m-1 \}, \, 2-n)/\{0,-1\}$  and for any $\gamma$
satisfying (\ref{a2.6}),
\begin{multline} \label{localentropyI}
\frac1{\alpha(\alpha+1)}\int_{\Omega} \zeta^4 u^{1+\alpha}(x, \, T)
\, dx +\\ c_{1} \Biggl[ \int_{P} \zeta^4 {u}^{\alpha + n -2 \gamma
+1} (u^{\gamma})_{xx}^2\,dx\,dt+ \int_{Q_T} \zeta^4 u^{\alpha+n-3} u_x^4 \, dx\,dt \Biggr]+ \\
\bigg[\int_{Q_T}\zeta^4u^{\alpha+m-1}u_x^2 \,dx\,dt +A \int_{Q_T}\zeta_x^4 u^{\alpha+M-1} u_x^2 \,dx\,dt \Biggr] \le\\
 c_{2} \int_{Q^T}( |\zeta_x|^4+ |\zeta \zeta_{xx}|^2) u^{n+\alpha
+1}\,dx\,dt+ c_{3} \int_{Q_T} |(\zeta^3 \zeta_x )_x| {u}^{\alpha + m + 1}\,dx\,dt+\\
\frac1{\alpha(\alpha+1)}\int_{\Omega} \zeta^4 {u_0}^{\alpha+1} \,
dx.
\end{multline}
\end{Theorem}

\begin{proof} Part $i)$ follows by setting $\zeta = 1$ in
(\ref{entropy-}),   implementing Lemma \ref{regularity}, and noting the H\"{o}lder regularity implied by the
boundedness of $\int_{Q_T} (u^{\frac{\alpha + m+1}{2}})_x^2 \, dx dt.$ Part
$ii)$  follows easily from (\ref{entropy-}) by letting $\epsilon
\rightarrow 0$ and noting that $\alpha + m +1$ and $\alpha + n+1$
are positive in the indicated parameter range.
\end{proof}

\begin{Remark} If $m<0$ and $m-n<-\frac{3}{2},$ {then}
the Infinite Speed of Propagation Property holds for the solutions
discussed in Theorem \ref{Case3a} due to the boundedness of
$\int_{Q_T} (u^{\frac{\alpha + m+1}{2}})_x^2 \, dx dt$. {See [20, Corollary 2.1].}
\end{Remark}

\begin{Theorem} \label{Case3b} {\bf{(The unstable case.)}} Let $\nu=1,$  $0 \le A,$
$0<n,$ $-\frac{3}{2} < m-n,$  $m-n<2$ if $A=0$, and $m<M$ if $0<A.$

\smallskip\par\noindent
$i)$ Let $\alpha_0=\alpha_0(\zeta \equiv 1),$ $\alpha_1=\max\{
\alpha_0, \, -2m+n-1 \},$ and $\beta_0= \frac{3}{n+
 \alpha_1 +1}.$   Then
 $u^{1/\beta}(\cdot, \, t)
\in {\mathcal{C}}^1([-a, \, a]),$ for all $\beta \in (0, \,
\beta_0)$ for almost every $t>0.$

\smallskip\par\noindent
$ii)$ For any $\zeta$ satisfying (\ref{Azeta}), let
 $\alpha_2=\max\{
\alpha_0(\zeta), \, -2m +n-1, \, -m-1 \}.$  Then, for any $\alpha \in (\alpha_2, \, 2-n)/\{0,-1\}$ and for
any $\gamma$ satisfying (\ref{a2.6}),
\begin{multline} \label{localentropyII}
\frac1{\alpha(\alpha+1)}\int_{\Omega} \zeta^4 u^{1+\alpha}(x, \,
T) \, dx +  A \int_{Q_T} \zeta^4 u^{\alpha + M-1} u_x^2 \,dx\,dt + \\
c_{1} \Biggl[ \int_{P} \zeta^4 {u}^{\alpha + n -2 \gamma +1}
(u^{\gamma})_{xx}^2\,dx\,dt + \int_{Q_T} \zeta^4
{u}^{\alpha +n-3} {u}_x^4 \,dx\,dt \Biggr] \le\\
 c_{2} \int_{Q^T}( |\zeta_x|^4+ |\zeta \zeta_{xx}|^2) u^{n+\alpha
+1} \,dx\,dt + c_{3} \int_{Q_T} |(\zeta^3 \zeta_x )_x| {u}^{\alpha + m + 1} \,dx\,dt +\\
c_{4} \int_{Q_T} \zeta^4 u^{\alpha + 2m -n +1}\,dx\,dt+
\frac1{\alpha(\alpha+1)}\int_{\Omega} \zeta^4 {u_0}^{\alpha+1} \,
dx.
\end{multline}

\end{Theorem}

\begin{proof} To prove part $i),$ note that the condition $-\frac{3}{2}<m-n$ implies that $\alpha_1 \in [\frac{1}{2}-n, \, 2-n),$ and that
$\alpha_1 \ge \alpha_0,$ $\alpha_1 + 2m-n+1 \ge 0,$ then implement
Lemma \ref{regularity} with $\zeta = 1.$ Part $ii)$  follows by
noting that $\alpha_2 \in [\frac{1}{2}-n, \, 2-n),$ $\alpha >
\alpha_0,$ $\alpha + m + 1 >0,$ and $\alpha + 2m -n +1 \ge 0,$ then
letting $\epsilon \rightarrow 0$ in (\ref{entropy+}).
\end{proof}

\begin{Remark} The
 results given in Theorem \ref{Case3b} also hold for the exceptional cases
 (\ref{gthermoc2}),
 (\ref{attractive}\,i), with $A=0$ and
 with $m-n$ assuming the values
$m-n=-1$ and $m-n=0,$ respectively. This can be easily demonstrated
by following the arguments above, once one notices that estimates
(\ref{Iap}), (\ref{Ibp}) also hold for (\ref{gthermoc2}),
(\ref{attractive}\,i) when the value of $m-n$ is taken as $-1$ or
$0,$ respectively, by utilizing the bounds (\ref{boundsa}),
(\ref{boundsb}).
 \end{Remark}

\begin{Remark} When $0<n<3,$  Corollary \ref{2strong} and (\ref{s17})
imply that under the assumptions in  Theorems \ref{Case3a} and \ref{Case3b},
the weak solutions obtained as the limit of a subsequence of $\{u_{\epsilon}(x,t)\},$ the solutions  to $({\mathbb{P}}_{\epsilon}),$
are in fact
strong solutions in the sense of Bernis  \cite{B1}.
In particular this implies  that solutions with compact support have zero contact angle at the edge of the
support.

 Using (\ref{s17m})--(\ref{s17}), one can check that in the context of Theorems \ref{Case3a} and \ref{Case3b},
\begin{equation} \label{s17p}
\beta_0 \ge \left\{ \begin{array}{ll} 2, &\quad 0<n
\le \frac{3}{2},\\ [1ex] \frac{3}{n}, & \quad \frac{3}{2} < n< 3.
\end{array} \right. \end{equation}
These bounds correspond to the bounds which were obtained in \cite{BBD95} for the thin film equation (\ref{thinfilm}).
The regularity indicated in (\ref{s17p}) also corresponds to the contact line regularity of the self-similar source type
solutions which were demonstrated in \cite{B97} to exist for (\ref{G}) when $0<n<3,$ $\nu=\pm 1,$ $A=0,$ and $m=n+2.$
In this sense the regularity obtained is sharp, though   additional regularity is provided by the terms $\int_{Q_T} (u^{\frac{\alpha + m+1}{2}})_x^2 \, dx,$
$\int_{Q_T} A (u^{\frac{\alpha + M+1}{2}})_x^2 \, dx,$ for some parameter values.
\end{Remark}

\begin{Remark}
For $0<n<3$, with regard to the terms in (\ref{weak}) which do not depend on $m$ or  $M$, since the regularity  obtained   is
the same as regularity which was obtained for (\ref{thinfilm}) in \cite{BBD95}, these terms can be interpreted in the same sense as in \cite{CPAM96}, which is distributional
for $\frac{3}{8}<n<3$ and is in the weaker sense of (\ref{weak}) for $0<n\le \frac{3}{8}.$ With regard to the
terms in (\ref{weak}) which depend on $m$ and $M$,
these terms may be readily seen to be interpretable in a suitable distributional sense if $m>-1$ and in the weaker sense of (\ref{weak}) for
$m\le -1$.  Details of the case $\nu=-1$, $0<m<1$ are discussed in \cite{BP94}.
\end{Remark}

In the case of "strong slippage," in which $0<n<2,$ the local
entropy estimates provided by Theorems \ref{Case3a}, \ref{Case3b}
can be used to prove the finite speed propagation property for the
strong solutions obtained there {when
 $m>\frac{n}{2}$, see \S 4. In the
case of "weak slippage" in which
\begin{equation}\label{weakslip}
2\leq n<3,
\end{equation}
{the finite speed propagation property for standard thin film
equation (1.1) and $u_0\in H^1(\Omega)$ was proved by F. Bernis
\cite{B2} using local energy estimates, which were  derived there.
 The proof of the energy estimates is based on certain
 integral inequalities for smooth
functions which are positive on some interval, see \cite {B6}. Our proof of the finite speed of propagation
property  for  (1.6) for $m>\frac{n}{2}$ in the case of "weak slippage"
(3.20),  is based on  combined
usage of  local entropy estimates, which were derived in Theorems \ref{Case3a} and \ref{Case3b}, and local energy estimates which are derived below in
 Theorem \ref{Th.le}. The derivation of the  local energy
estimates makes use of a generalization of the Bernis inequalities mentioned earlier, for
nonnegative periodic functions, see Lemma 6.3}.} Somewhat similar
methodologies were employed in \cite{HSh,GS} where local energy
estimates were also derived, and {which also relied on combined usage  of  local entropy
and energy estimates.}
 Notably, the combination of local energy and energy estimates used here and in \cite{GS} relies on using local entropy estimates
with $-1 < \alpha <0$; previous approaches relied on local entropy estimates with $\alpha >0$ in proving the finite speed of propagation property \cite{BBD95,CPAM96,DGS01}.

\begin{Theorem}\label{Th.le}
Suppose that $\nu=\pm1,$ $2 \le n<3,$ $0 \le  A,$
 and that $m<M$ if $A>0,$ and $m<n+2$ if $A=0$ and $\nu=1.$
Suppose moreover if $\nu=-1,$ that  $m- \frac{3}{4}n \ge -1,$
and  if $\nu=1,$ that  $m-\frac{2}{3}n>-\frac{2}{3}$ if $2 \le n< \frac{5}{2}$ and $m-n
> -\frac{3}{2}$ if $\frac{5}{2} \le n <3.$
Under these assumptions,  the strong solutions obtained in Theorems \ref{Case3a},
\ref{Case3b} satisfy the following local energy estimate
\begin{multline}\label{62}
\int_\Omega\zeta^6|u_x(x,T)|^2\,dx+d_{1}\int_{Q_T}\zeta^6\big((u^{\frac{n+2}6})_x^6
+(u^{\frac{n+2}3})_{xx}^3+(u^{\frac{n+2}2})_{xxx}^2\big)\,dx\,dt+\\d_{1}\int_{Q_T\cap\{u>0\}
} \zeta^6u^nu_{xxx}^2\,dx\,dt\leq\int_\Omega\zeta^6|u_{0x}(x)|^2\,dx+\\
d_{2}\int_{Q_T}u^{n+2}(|\zeta_x|^6\,dx\,dt+|\zeta\zeta_{xx}|^3)\,dx\,dt-\\\int_{Q_T\cap\{u>0\}}(\nu
u^mu_x-Au^Mu_x)(u_x\zeta^6)_{xx}\,dx\,dt,
\end{multline}
where  $\zeta(x)$ is an arbitrary nonnegative function from
$C^4([-a,a])$.
\end{Theorem}

\begin{proof}
In order to derive the local energy estimate, we test the equation in the approximating problem, $(\mathbb
P_\epsilon),$ with $-(\zeta^6{u_\epsilon}_x)_x,$ and easily deduce
that
\begin{multline}\label{51}
\int_\Omega\frac{\zeta^6}{2}|{u_\epsilon}_x(x,T)|^2\,dx+\int_{Q_T}\zeta^6f_\epsilon(u_\epsilon)
|{u_\epsilon}_{xxx}|^2 \,dx\,dt=\\-\int_{Q_T}
f_\epsilon(u_\epsilon){u_\epsilon}_{xxx}[2{u_\epsilon}_{xx}(\zeta^6)_x+{u_\epsilon}_x
(\zeta^6)_{xx}]\,dx\,dt-\\\int_{Q_T} f_\epsilon(u_\epsilon)(\nu
u_\epsilon^{m-n}{u_\epsilon}_x-Au_\epsilon^{M-n}{u_\epsilon}_x)({u_\epsilon}_x\zeta^6)_{xx}\,dx\,dt+\\
\int_\Omega\frac{\zeta^6}{2}|u_{0\epsilon x}|^2\,dx.
\end{multline}

Assuming (\ref{weakslip}) and noting (\ref{s17m}), (\ref{s17}), it
is easy to check that $\max\{ \alpha^{\ast}, \, -m-1 \} <0$ in  Theorem \ref{Case3a}   and  that  $\max\{ \alpha^{\ast}, \, -2m + n -1, \, -m -1 \} <0$
in Theorem
\ref{Case3b}.
Hence if $n$ satisfies (\ref{weakslip}), then for arbitrary $\zeta$ satisfying (\ref{Azeta}) or
$\zeta = 1,$ $\alpha$ may be chosen to
be fixed and to satisfy
 \begin{equation}\label{52}
-1<\alpha<0,
\end{equation}
in addition to satisfying the constraints indicated in Part $ii)$ of
either theorem.

For $\alpha$ satisfying \eqref{52}, we have due to
(\ref{massconservation})
\begin{multline}\label{53}
\int_\Omega u_\epsilon(x,T)^{1+\alpha}\,dx\leq
|\Omega|^{-\alpha}\bigg(\int_\Omega
u_\epsilon(x,T)\,dx\bigg)^{1+\alpha}=\\
|\Omega|^{-\alpha}\bigg(\int_\Omega
u_{0\epsilon}\,dx\bigg)^{1+\alpha}\leq d_{3}.
\end{multline}
 Setting
$\zeta = 1$ in (\ref{entropy-}) and employing
(\ref{icr}),(\ref{53}), we obtain the following inequality when
$\nu=-1,$
\begin{multline}\label{54}
c_3
\Biggl[\int_{Q_T}u_\epsilon^{\alpha+n-2\gamma+1}(u_\epsilon^\gamma)_{xx}^2 \,dx\,dt
+\int_{Q_T}u^{\alpha+n-3}_{\epsilon}{u_\epsilon}_ x^4\,dx\,dt\Biggr]+\\
\int_{Q_T}(u_\epsilon^{\alpha+m-1}+Au_\epsilon^{\alpha+M-1}){u_\epsilon}_
x^2 \,dx\,dt \leq  \int_\Omega G_\epsilon(u_{0\epsilon})\,dx+d_{4}\leq
d_{5}.
\end{multline}
Similarly, setting $\zeta = 1$ in (\ref{entropy+}), employing
(\ref{icr}),(\ref{53}), and noting that the assumptions on $\alpha$
imply that $\alpha > \alpha_2 \ge -2m+n-1,$ we obtain the following
inequality when $\nu=1,$
\begin{multline}\label{55}
c_8 \Biggl[
\int_{Q_T}u_\epsilon^{\alpha+n-2\gamma+1}(u_\epsilon^\gamma)^2_{xx} \,dx\,dt +\int_{Q_T}u_\epsilon^{\alpha+n-3}
{u_\epsilon
}_x^4 \,dx\,dt\Biggr] +\\ \int_{Q_T} A u_\epsilon^{\alpha+M-1}{u_\epsilon}_x^2 \,dx\,dt\leq
 \int_\Omega G_\epsilon(u_{0\epsilon})\,dx+c_{5}
\int_{Q_T}u_\epsilon^{\alpha+2m-n+1}\,dx\,dt+ d_{6} \leq d_{7}.
\end{multline}

Next we pass to the limit $\epsilon\to0$ in \eqref{51}. First,
setting $\zeta = 1$ in  \eqref{51} yields that for any $\delta>0,$
\begin{multline}\label{56}
2^{-1}\int_\Omega|{u_\epsilon}_x(x,T)|^2\,dx+\int_{Q_T}
f_\epsilon(u_\epsilon)|{u_\epsilon}_{xxx}|^2\,dx\,dt\leq
2^{-1}\int_\Omega|u_{0\epsilon x}
|^2\,dx+\\
(1+A|\sup
u_\epsilon|^{M-m})\int_{Q_T}f_\epsilon(u_\epsilon)u_\epsilon^{m-n}|{u_\epsilon}_x||{u_\epsilon}_{xxx}|\,dx\,dt\leq
2^{-1}\int_\Omega|u_{0\epsilon x}|^2\,dx+\\
\delta\int_{Q_T}f_\epsilon(u_\epsilon)|{u_\epsilon}_{xxx}|^2\,dx\,dt+\int_{Q_T}u_\epsilon^{\alpha+n-3}
{u_\epsilon}_x^4\,dx\,dt+c(\delta)\int_{Q_T}u_\epsilon^{4m-2n-(\alpha+n)+3}\,dx\,dt.
\end{multline}
Suppose that
\begin{equation}\label{57}
4m-2n-(\alpha+n)+3\geq0,
\end{equation}
then  setting $\delta=1/2$ in \eqref{56} and  using the estimates
\eqref{54}, \eqref{55}, we deduce that
\begin{equation}\label{59}
\int_\Omega|{u_\epsilon}_x(x,T)|^2\,dx+\int_{Q_T}f_\epsilon(u_\epsilon)|{u_\epsilon}_{xxx}|^2\,dx\,dt<d_{8}.
\end{equation}
\begin{Remark}\label{r:3.8}
In the context  of Theorem \ref{Case3a}, when
$n$ satisfies (\ref{weakslip}), $m > n-2 \ge 0,$ hence $\max\{
\alpha^{\ast}, \, -m -1 \}= \alpha^{\ast}= -1.$ Thus for arbitrary
initial data satisfying (\ref{ic}),(\ref{icr}), (\ref{57}) is
satisfied for some admissible $\alpha,$ if
$$ m - \frac{3}{4} n  \ge -1,$$ which is stronger than the
previous constraint, $m-n>-2$.
In the context  of Theorem \ref{Case3b}, when $n$
satisfies (\ref{weakslip}), $m>n-2 \ge 0;$ hence $\max\{
\alpha^{\ast}, \, -m-1 \} = \alpha^{\ast}=-1,$ but $\alpha$ there must
also satisfy $\alpha > -2m + n -1.$
 It is easy to
check that  \eqref{57} holds for some admissible $\alpha,$ if and
only if $m,n$ satisfy the condition
\begin{equation}\label{58}
3m-2n>-2.
\end{equation}
Recalling the constraint $m-n>-\frac {3}{2}$ in Theorem
\ref{Case3b}, it
  is easy to check that
 \eqref{58} constitutes an additional constraint if $n<\frac52.$
\end{Remark}

Using the estimates \eqref{59}, \eqref{54}, \eqref{55}, it is easy
to check that the integrals on the right-hand side of \eqref{51}
are uniformly bounded with respect to $\epsilon$ if  \eqref{57} is
satisfied. For arbitrary $\eta>0,\ u_\epsilon\to u$ strongly in
the space $C^{4,1}(\{u>\eta\})$. Therefore passage to the limit
$\epsilon\to0$ in all of the  integrals in \eqref{51} over the
domain $\{u>\eta\}$ is straightforward. As to integrals over the
domain $\{u<\eta\},$ we have, for example, by virtue of
\eqref{57},
\begin{multline*}
\bigg|\int_{Q_T\cap\{u<\eta\}}f_\epsilon(u_\epsilon){u_\epsilon}_{xxx}{u_\epsilon}_x
{u_\epsilon}^{m-n}
\zeta^6 \,dx\,dt\bigg|\leq\\
\bigg(\int_{Q_T\cap\{u<\eta\}}f_\epsilon(u_\epsilon)|{u_\epsilon}_{xxx}|^2\,dx\,dt\bigg)^{1/2}\bigg(
\int_{Q_T\cap\{u<\eta\}}u_\epsilon^{\alpha+n-3}{u_\epsilon}_x^4\,dx\,dt\bigg)^{1/4}\times\\\bigg(
\int_{Q_T\cap\{u<\eta\}}u^{4m-2n-(\alpha+n)+3}\,dx\,dt\bigg)^{1/4}\leq
c\eta^{\frac{4m-2n-(\alpha+n)+3}{4}}\to0\text{ as } \eta\to0.
\end{multline*}
Analogously, it is easy to check that all of the other integrals
over $\{u<\eta\}$ on the right-hand side of \eqref{51} are bounded
from above by some continuous function, $h(\eta),$ such that $\
h(\eta)\to0$ as $\eta\to 0$. Therefore, first passing to the limit
$\epsilon\to0,$ and afterwards letting  $\eta\to0,$ we easily
obtain
\begin{multline}\label{60}
2^{-1}\int_\Omega\zeta^6|u_x(x,T)|^2\,dx+\int_{Q_T\cap\{u>0\}}\zeta^6u^nu_{xxx}^2\,dx\,dt\leq\\
2^{-1}\int_\Omega
\zeta^6|u_{0x}|^2\,dx-\int_{Q_T\cap\{u>0\}}u^nu_{xxx}[2u_{xx}(\zeta^6)_x+u_x(\zeta^6)_{xx}]\,dx\,dt-\\
\int_{Q_T\cap\{u>0\}}(\nu u^mu_x-Au^Mu_x)(u_x\zeta^6)_{xx}\,dx\,dt.
\end{multline}

Since $u(\cdot,t)\in C^1(\bar \Omega)$ for almost $t\in[0,T]$, it
is possible to estimate from below the second term on the left
hand side of \eqref{60}  using the generalized Bernis
inequalities given in Lemma 6.3. As result we obtain:
\begin{multline}\label{61}
2^{-1}\int_\Omega\zeta^6|u_x(x,T)|^2\,dx+d_{7}\int_{Q_T}\zeta^6\big((u^{\frac{n+2}6})_x^6
+|(u^{\frac{n+2}3})_{xx}|^3+(u^{\frac{n+2}2})_{xxx}^2\big)\,dx\,dt
+\\d_9\int_{Q_T\cap\{u>0\}}\zeta^6u^nu_{xxx}^2\,dx\,dt\leq2^{-1}\int_\Omega\zeta^6|u_{0x}(x)|^2\,dx+
d_{10}\int_{Q_T}|\zeta_x|^6u^{n+2}\,dx\,dt-\\\int_{Q_T\cap\{u>0\}}u^nu_{xxx}[2u_{xx}(\zeta^6)_x
+u_x(\zeta^6)_{xx}]\,dx\,dt-\\\int_{Q_T\cap\{u>0\}}(\nu
u^mu_x-Au^Mu_x)(u_x\zeta^6)_{xx}\,dx\,dt,
\end{multline}

Next we estimate the terms in the third integral on the right-hand
side as in \cite{HSh}; namely,
\begin{multline*}
\int_{Q_T\cap\{u>0\}}u^nu_{xxx}u_x(\zeta^6)_{xx}\,dx\,dt\\=6\int_{Q_T\cap\{u>0\}}
(u^{\frac
n2}u_{xxx}\zeta^3)(u^{\frac{n-4}6}u_x\zeta)(u^{\frac{n+2}3}(5\zeta_x^2+\zeta\zeta_{xx}))\,dx\,dt\\\leq
6\bigg(\int_{Q_T\cap\{u>0\}}
u^nu_{xxx}^2\zeta^6\,dx\,dt\bigg)^{1/2}
\bigg(\int_{Q_T\cap\{u>0\}}u^{n-4}u_x^6\zeta^6\,dx\,dt\bigg)^{1/6}\\\times
\bigg(\int_{Q_T}u^{n+2}(5\zeta_x^2+\zeta\zeta_{xx})^3\,dx\,dt\bigg)^{1/3}\leq
\delta\int_{Q_T\cap\{u>0\}}
(u^nu_{xxx}^2+(u^{\frac{n+2}6})_x^6)\zeta^6\,dx\,dt\\+c(\delta)\int_{Q_T}u^{n+2}(\zeta_x^6+(\zeta\zeta_{xx})^3
)\,dx\,dt, \qquad\forall\,\delta>0,\hskip 50 pt
\end{multline*}
and
\begin{multline*}
\int_{Q_T\cap\{u>0\}}u^nu_{xxx}u_{xx}(\zeta^6)_x\,dx\,dt=6\int_{Q_T\cap\{u>0\}}
(u^{\frac
n2}u_{xxx}\zeta^3)\\\times\Big[\Big(u^{\frac{n-1}3}u_{xx}+\frac{n-1}3u^{\frac{n-4}3}u_x^2-\frac{n-1}3
u^{\frac{n-4}3}u_x^2\Big)\zeta^2\Big](u^{\frac
n2-\frac{n-1}3}\zeta_x)\,dx\,dt\\\leq 6\bigg(\int_{Q_T\cap\{u>0\}}
u^nu_{xxx}^2\zeta^6\,dx\,dt\bigg)^{1/2}\bigg[\int_{Q_T\cap\{u>0\}}
\Big(\frac3{n+2}|(u^{\frac{n+2}3})_{xx}|\\+\frac{n-1}3\Big(\frac6{n+2}\Big)^2(u^{\frac{n+2}6})_x^2
\Big)^3\zeta^6 \,dx\,dt\bigg]^{1/3}\bigg(\int_{Q_T}u^{n+2}\zeta_x^6\,dx\,dt\bigg)^{1/6}\\\leq
\delta\int_{Q_T\cap\{u>0\}}
(u^nu_{xxx}^2+|(u^{\frac{n+2}3})_{xx}|^3+(u^{\frac{n+2}6})_x^6)\zeta^6\,dx\,dt\\
+c(\delta)\int_{Q_T}u^{n+2}\zeta_x^6\,dx\,dt,\qquad\forall\,\delta>0.
\end{multline*}

Using these estimates in  \eqref{61} with
$\delta=d_9/6,$ we obtain (\ref{62}).
\end{proof}

\section{Finite speed propagation (strong slippage: $0<n<2$)}
\setcounter{equation}{0}

In this section, we consider  problem $(\mathbb P)$ with initial
data, $u_0$, which satisfies \eqref{ic} and which also possesses
the additional property,
\begin{equation}\label{4.4}
\supp u_0\subset\{x\leq0\}.
\end{equation}
Let us introduce the following family of subdomains:
\begin{equation}\label{4.5}
\begin{aligned}
&\Omega(s)=\Omega\cap\{x\,|\, x>s\}\quad\forall\,s\in(-a,a),\qquad
&Q_{t}(s)=\Omega(s)\times(0,t).
\end{aligned}
\end{equation}

\begin{Theorem}\label{t:4.1}
Let  $u_0$ satisfy \eqref{ic}, \eqref{4.4},  let $\nu=\pm1,$ $0
\le A,$ $0<n<2,$ $m<M$ if $0<A,$ $ m<n+2$ if $\nu=1,$ $A=0,$ and
$$ m>0\quad\text{if}\quad \nu=-1,\qquad m>\frac
n2\quad\text{if}\quad \nu=1,$$
 and let $u$ denote an
arbitrary strong nonnegative solution of problem $(\mathbb P)$,
obtained as in Theorem 1, which satisfies the local entropy
estimate in Theorem 2 or 3. Then $u$ possesses the finite speed of
propagation property in the sense that there exists a continuous
function, $s(t),$  such that $\ s(0)=0$, and a positive time
$T_0,$ such that
\begin{equation}\label{4.7}
\supp u(\cdot,t)\subset\overline\Omega\setminus\Omega(s(t)),
\qquad s(t)<a \quad \forall\,t<T_0,\ s(T_0)=a.
\end{equation}
\end{Theorem}

\begin{Remark} The analysis in the section  applies also
to (\ref{attractive}\hbox{\,i\,}).
\end{Remark}

\begin{proof}
Since by assumption $0<n<2,$ the local entropy estimate,
\eqref{localentropyI} or \eqref{localentropyII},  holds for some
positive $\alpha<2-n$. The proof of the finite speed of propagation
property is based on careful analysis of the properties of solutions
satisfying these inequalities with  positive $\alpha$. In the stable
case, $\nu=-1,$ such analysis was performed  in \cite{DGS01} with
$A=0$. If $\nu=-1$ and $A>0,$ a similar proof can be given.
Therefore we restrict our attention here to  the unstable case,
$\nu=1.$ Although we set $A=0$ for simplicity,  all our estimates
are also valid for the case $\nu=1,\ A>0.$  Thus, the estimate
\eqref{localentropyII} can be written in the form
\begin{multline}\label{4.8}
\frac1{\alpha(\alpha+1)}\int_\Omega \zeta^4u^{1+\alpha}(x,T)\,dx +
c_{1}\int_{Q_T}
\zeta^4\bigl(u^{\alpha+n-2\gamma+1}(u^\gamma)^2_{xx}+u^{\alpha+n-3}u_x^4\bigr)\,dx\,dt
\\\leq\frac1{\alpha(\alpha+1)}
\int_\Omega\zeta^4u_0^{1+\alpha}\,dx+cR,\qquad\alpha>0,
\end{multline}
where the constant $c=\max(c_{2},c_{3},c_{4})$ and
\begin{multline*}
R:=R_1+R_2+R_3=\int_{Q_T}(|\zeta_x^4|+|\zeta\zeta_{xx}|^2)u^{n+\alpha+1}\,dx\,dt+\\
\int_{Q_T}|(\zeta^3\zeta_x)_x|u^{\alpha+m+1}\,dx\,dt+\int_{Q_T}\zeta^4u^{\alpha+2m-n+1}\,dx\,dt.
\end{multline*}
Note now that due to the constraint $m>n/2,$ the
 minimal power  of
$u(x,t)$ in terms of the form $\int_{Q_T} \zeta^4 u^{\sigma_1} \, dx dt$ on the right hand side  of  (\ref{4.8}) is
strictly greater than the minimal power of $u$ in terms of the  form $\int_{\Omega} \zeta^4 u^{\sigma_2} \, dx$ appearing on the left hand side of  (\ref{4.8});
namely,
$$
\alpha + 2m-n+1>\alpha+1\Longleftrightarrow m>\frac{n}{2}.
$$
{This shall allow us to derive a  functional relationship  from (\ref{4.8}), for a suitable set of energy
functions,  which satisfies} the conditions of the
generalized Stampacchia lemma, Lemma 6.2.

To put (\ref{4.8}) in an appropriate format, we  set $\zeta(x)=\zeta_{s,\delta}(x)$, where
   $\zeta_{s,\delta}(x)$ is the cut-off function defined as
\begin{equation}\label{4.13}
\zeta_{s,\delta}(x)=\varphi\Bigl(\frac{x-s}\delta\Bigr),
\end{equation}
where $s\in\mathbb R, \delta>0$ are free parameters, and
$\varphi(r)$ is a nonnegative nondecreasing $C^2(\mathbb R)$
function such that:
\begin{equation}\label{4.14}
\varphi(r)=0\text{ for } r\leq0,\quad \varphi(r)=1\text{ for } r>1.
\end{equation}

We now define  three energy functions {which are}
 connected with the terms on the right hand side
of our entropy estimate \eqref{4.8},
\begin{equation}\label{4.15}
\begin{aligned}
J_T(s)&:=\int_{Q_T(s)}u^{\beta_1+\alpha+1}\,dx\,dt,
E_T(s):=\int_{Q_T(s)}u^{\beta_2+\alpha+1}\,dx\,dt,\\
I_T(s)&:=\int_{Q_T(s)}u^{\beta_3+\alpha+1}\,dx\,dt,
\end{aligned}
\end{equation}
where $\beta_1=n,\ \beta_2=m, \ \beta_3=2m-n$. Using
 \eqref{4.8}, we {shall} deduce {a system of} three functional inequalities
{for}  $J_T(s),$ $E_T(s),$  $I_T(s).$ {Having set}
$\zeta(x)=\zeta_{s,\delta}(x)$ in \eqref{4.8}, we obtain after
some simple computations
\begin{multline}\label{4.16}
\sup_{t\in(0,T)}\int_{\Omega(s+\delta)}u^{\alpha+1}\,dx+D_1
\int_{Q_{T}(s+\delta)}
\Bigl[\bigl(u^{\frac{\alpha+n+1}2}\bigr)_{xx}^2+\bigl(u^{\frac{\alpha+n+1}4}\bigr)_x^4\Bigr]\,dx\,dt\\
\leq D_2 \Biggl[\frac1{\delta^4}\int_{Q_T(s)\setminus
Q_T(s+\delta)}
u^{\alpha+n+1}\,dx\,dt+\frac1{\delta^2}\int_{Q_T(s)\setminus
Q_T(s+\delta)}
u^{\alpha+m+1}\,dx\,dt+\\\int_{Q_T(s)}u^{\alpha+2m-n+1}\,dx\,dt\Biggr]+\int_{\Omega(s)}u_0^{1+\alpha}dx.
\end{multline}
Here and throughout the proof, $D_i$  denote  positive constants
which can depend on the problem  parameters, $\alpha,$ $n,$ $m,$
but not
 on $s,$ $\delta,$ and $T$. For arbitrary $\beta>0,\ l>4,\ s>0,
\delta>0,$ such that $s+2\delta<a,$ we obtain for $\sigma=\beta+\alpha+1$ that
\begin{equation}\label{4.17}
\int_{\Omega(s+2\delta)}u^{\sigma}dx\leq\int_{\Omega(s+\delta)} (u
\, \zeta_{s+\delta,\delta}^l)^{\sigma}dx=\int_{\Omega(s+\delta)}
v^{\frac{4 \sigma}{n+\alpha+1}}dx,
\end{equation}
where
$v=v(x,t)=(u\,\zeta_{s+\delta,\delta}^l)^{\frac{n+\alpha+1}4}$. By
the Gagliardo--Nirenberg interpolation inequality,
\begin{equation}\label{4.18}
\int_{\Omega(s+\delta)}v^{\frac{4\sigma}{n+\alpha+1}}dx\leq D_3
\Biggl[\int_{\Omega(s+\delta)}|v_x|^4dx\Biggr]^{\frac{\theta\sigma}{n+\alpha+1}}
\Biggl[\int_{\Omega(s+\delta)}
v^{\frac{4(\alpha+1)}{n+\alpha+1}}dx\Biggr]^{\frac{(1-\theta)\sigma}{\alpha+1}}
\end{equation}
where $\theta=\frac{\beta(n+\alpha+1)}{\sigma(n+4(\alpha+1))}$.
Combining \eqref{4.17} and \eqref{4.18},
\begin{multline}\label{4.19}
\int_{\Omega(s+2\delta)}u^{\beta+\alpha+1}dx\leq D_4
\Biggl[\int_{\Omega(s+\delta)}u^{\alpha+1}dx\Biggr]^{\frac{(1-\theta)\sigma}{\alpha+1}}\\\times
\Biggl[\int_{\Omega(s+\delta)}\bigl|\bigl(u^{\frac{n+\alpha+1}4}\bigr)_x\bigr|^4dx
+\frac1{\delta^4}\int_{\Omega(s+\delta)\setminus\Omega(s+2\delta)}
u^{n+\alpha+1}dx\Biggr]^{\frac{\theta\sigma}{n+\alpha+1}}.
\end{multline}
Let us suppose that
$$\frac{\theta\sigma}{n+\alpha+1}<1.$$
or, equivalently, {that
\begin{equation}\label{4.20}
\beta<n+4(\alpha+1).
\end{equation}}
 Then,
integrating \eqref{4.19} with respect to  $t$ and using H\"older's
inequality,
\begin{multline}\label{4.21}
\int_{Q_{T}(s+2\delta)} u^{\beta+\alpha+1}\,dx\,dt\leq D_5
T^{1-\frac{\beta}{n+4(\alpha+1)}}\sup_{t\in(0,T)}\Biggl(\int_{\Omega(s+\delta)}
u^{\alpha+1}(t)\,dx\Biggr)^{\frac{(1-\theta)(\beta+\alpha+1)}{\alpha+1}}\\
\times\Biggl[\frac1{\delta^4}\int_{Q_{T}(s+\delta)\setminus
Q_{T}(s+2\delta)}u^{n+\alpha+1}\,dx\,dt+\int_{Q_{T}(s+\delta)}
\bigl|\bigl(u^{\frac{n+\alpha+1}4}\bigr)_x\bigr|^4\,dx\,dt
\Biggr]^{\frac\beta{n+4(\alpha+1)}}\\
\\\leq
D_6
T^{1-\frac\beta{n+4(\alpha+1)}}\Biggl[\sup_{t\in(0,T)}\int_{\Omega(s+\delta)}
u^{\alpha+1}(t)\,dx+\frac1{\delta^4}\int_{Q_{T}(s+\delta)\setminus
Q_{T}(s+2\delta)} u^{n+\alpha+1}\,dx\,dt\\
+\int_{Q_{T}(s+\delta)}\bigl|\bigl(u^{\frac{n+\alpha+1}4}\bigr)_x\bigr|^4\,dx\,dt\Biggr]^{1+\mu},
\end{multline}
where $\mu=\frac{4\beta}{n+4(\alpha+1)}>0$.

Using the definitions in  \eqref{4.15} and the a priori estimate
\eqref{4.16}, we deduce from  \eqref{4.21} that
\begin{multline}\label{4.22}
\int_{Q_{T}(s+2\delta)} u^{\beta+\alpha+1} \, dx\, dt\leq D_7
T^{1-\frac\mu{4}}\Bigl[\frac{J_T(s)-J_T(s+2\delta)}{\delta^4}+\\
\frac{E_T(s)-E_T(s+\delta)}{\delta^2}+I_T(s)+\int_{\Omega(s)}u_0^{\alpha+1}\,dx\Bigr]^{1+\mu}.
\end{multline}
The inequality \eqref{4.22} holds for the three values of
$\beta_i,\ i=1,2,3$, prescribed in \eqref{4.15}, if  condition
\eqref{4.20} holds with $\beta=\beta_i,\ i=1,2,3.$ These
conditions may be written as
$$
1)\ \alpha+1>0, \quad 2)\ m<n+4(\alpha+1),\quad 3)\
2m-n<n+4(\alpha+1).
$$
It is easy to check that all of these conditions are satisfied for
some $\alpha\in(\alpha_2,2-n)$ if and only if
\begin{equation}\label{4.25}
m<6-n.
\end{equation}

Thus, if {the} inequality \eqref{4.25} holds, we obtain the following
system of functional inequalities:
\begin{equation}\label{4.26}
\begin{aligned}
 J_T(s+\delta)\leq D_8 T^{\frac{4-
\mu_1}{4}}\Bigl[\frac{J_T(s)-J_T(s+\delta)}{\delta^4}+
&\frac{E_T(s)-E_T(s+\delta)}{\delta^2}+\\&I_T(s)+h_0(s) \Bigr]^{1+\mu_1}, \\
 E_T(s+\delta)\leq D_9 T^{\frac{4-\mu_2}{4}}\Bigl[
\frac{J_T(s)-J_T(s+\delta)}{\delta^4}+&\frac{E_T(s)-E_T(s+\delta)}{\delta^2}+\\&
I_T(s)+h_0(s)\Bigr]^{1+\mu_2}, \\
 I_T(s+\delta)\leq
D_{10}
T^{\frac{4-\mu_3}{4}}\Bigl[\frac{J_T(s)-J_T(s+\delta)}{\delta^4}+
&\frac{E_T(s)-E_T(s+\delta)}{\delta^2}+\\&I_T(s)+h_0(s) \Bigr]^{1+\mu_3}, \\
\end{aligned}
\end{equation}
where $h_0(s)=\int_{\Omega(s)}u_0^{1+\alpha}\,dx$, and
$$\mu_1=\frac{4n}{n+4(\alpha+1)},\
\mu_2=\frac{4m}{n+4(\alpha+1)},\
\mu_3=\frac{4(2m-n)}{n+4(\alpha+1)}.$$ Due to the boundedness and
nonnegativity of $u,$ the following estimates are obvious
\begin{equation}\label{4.27}
\begin{aligned}
& J_T(0)\leq J_T:=\int_{Q_T}u^{\beta_1+\alpha+1}\,dx\,dt<cT,\quad\forall\,T>0, \\
& E_T(0)\leq E_T:=\int_{Q_T}u^{\beta_2+\alpha+1}\,dx\,dt<cT,\quad\forall\,T>0, \\
& I_T(0)\leq
I_T:=\int_{Q_T}u^{\beta_3+\alpha+1}\,dx\,dt<cT,\quad\forall\,T>0,
\end{aligned}
\end{equation}
where $c$ is a constant which  does not depend on $T$.  The
validity of the statement of Theorem \ref{t:4.1} when inequality
\eqref{4.25} holds, now follows from \eqref{4.26}, \eqref{4.27},
and {Lemma  6.2}, since $h_0(s)=0$ for any $s>0$.

If $m\geq 6-n,$ we proceed as follows.  {We fix}  $\overline{m}$ such that $
\frac n2<\overline{m}<6-n$. It is easy to see that due to the
boundedness of the solution $u,$ all of the previous estimates in
the proof of Theorem \ref{t:4.1}  remain {valid} when $m$ is
replaced by $\overline{m}$. As {a} result,   the system \eqref{4.26} is
obtained with respect to new energy functions \eqref{4.15} defined
by the values $$\beta_1=n,\ \beta_2=\overline{m},\ \beta_3=2\,\overline{m}-m,$$
and where
$$
\mu_1=\frac{4n}{n+4(\alpha+1)},\quad \mu_2=\frac{4\,\overline{m}}{n+4(\alpha+1)},
\quad \mu_3=\frac{4(2\,\overline{m}-n)}{n+4(\alpha+1)}.
$$
 In this manner, the
validity of the statement of the theorem for the case $m \ge 6-n$
again follows from (\ref{4.26}) and {Lemma  6.2}.
\end{proof}

\section{Finite speed propagation (weak slippage: $2\leq n<3$)}

\setcounter{equation}{0}

In this section we shall again  consider  problem $(\mathbb P),$ with
initial data, $u_0$, which satisfies \eqref{ic} as well as the
additional property, \eqref{4.4}. The subdomains, $\Omega(s)$ and
$Q_{t}(s),$ {are to} be understood here to be as defined in \eqref{4.5}.

We first prove the following lemma, which provides control on the
$L_{loc}^1(\Omega)$ norm of some minimal positive power of the
solution under consideration, $u(x,t)$. For the sake of simplicity,
the results in this section are proven for  $\nu=1$ and $A=0,$
though they remain valid for $\nu=-1$ and $A>0$ as well. The results
here can also be readily shown to apply to (\ref{gthermoc2}) if
$2<n<3$ and to (\ref{attractive}\,i).

\begin{Lemma}\label{l:5.1}
Let $\nu=\pm 1,$ $0\le A,$  $1/2< n<3,$ $m>n/2,$ $\eta>\frac{1-n}3,$
$\varepsilon>0,$ with $m<M$ if $A>0$ and $m<n+2$ if $\nu=1,$ $A=0$.
Then there exists a positive constant $c$, depending on
$n,m,\eta,\varepsilon$ only, such that any nonnegative strong
solution $u$ of problem $(\mathbb P)$ satisfies
\begin{multline}\label{5.1}
\int_\Omega u(x,T)^{\eta+1}\zeta^4\, dx \leq \varepsilon\Biggl(\int_{Q_T
\cap \{u>0\}}\zeta^6u^nu_{xxx}^2 \, dx\, dt
+\int_{Q_{T}}|\zeta_x|^6u^{n+2}\, dx\,dt \Biggr)+\\c\Biggl(\int_{Q_T}\bigl[
u^{n+2\eta}|\zeta_x|^2+u^{\frac{3m+3\eta+1-n}2}\zeta^3
+u^{m+\eta+1}|\zeta\zeta_x|^2\bigr] \, dx\,dt \Biggr)+\\
\int_\Omega u_0^{\eta+1}\zeta^4 \, dx
+c\int_{Q_T\cap\supp\zeta}u^{n+3\eta-1}\, dx\,dt,
\end{multline}
for arbitrary nonnegative $\zeta\in C^2([-a,a])$.
\end{Lemma}
\begin{proof}
The proof  here follows that of  Lemma 5.2 in \cite{GS}. Let us
{set $\phi=\varphi$ in} the integral identity in \eqref{weak}, {where $\varphi$ is the} test
function
$$
\varphi=-l_\delta\zeta^4(u+\gamma)^\eta,\quad \gamma>0,
$$
where $\{l_\delta(t)\}\subset C_c^\infty(0,T)$ and
$l_\delta\to\chi_{(0,T)}$ as $\delta\to 0$. After some simple
computations, we obtain
\begin{multline}\label{5.1*}
-\int_{Q_T}(l_\delta)_t\zeta^4\frac{(u+\gamma)^{\eta+1}}{\eta+1} \, dx\,dt =\int_{Q_T}
u^mu_xl_\delta((u+\gamma)^\eta\zeta^4)_x \,dx\,dt +\\
\int_{Q_T\cap\{u>0\}}l_\delta(\zeta^4)_xu^n(u+\gamma)^\eta u_{xxx}\, dx\,dt +\\
\eta\int_{Q_T\cap\{u>0\}}l_\delta\zeta^4u^n(u+\gamma)^{\eta-1}u_xu_{xxx}\, dx\,dt:=A_1+A_2+A_3.
\end{multline}
The terms $A_2,A_3$ {can be} estimated  as in \cite[Lemma 5.2]{GS}. For
any $\epsilon >0,$
$$
|A_2|\leq\epsilon\int_{Q_T\cap\{u>0\}}\zeta^6u^nu_{xxx}^2 \, dx\,dt +C_1(\epsilon)\int_{Q_T}|\zeta_x|^2
u^n(u+\gamma)^{2 \eta}\, dx\,dt,
$$
\begin{multline*}
|A_3|\leq\epsilon\bigg(\int_{Q_T\cap\{u>0\}}\zeta^6u^nu_{xxx}^2\, dx\,dt +\int_{Q_T}|\zeta_x|^6u^{n+2}\, dx\,dt \bigg)
+\\C_2(\epsilon)\int_{Q_T\cap \supp\zeta}(u+\gamma)^{n+3\eta-1}\, dx\,dt.
\end{multline*}
Here $C_i$ denote constants which may depend on $m,$ $n,$ $\eta,$
and on $\epsilon$ if indicated, but which are independent of
$\gamma$ and $\delta$.

Let us now estimate $A_1$.
\begin{multline}\label{5.2}
A_1=\int_{Q_T} \eta u^m(u+\gamma)^{\eta-1}u_x^2\zeta^4l_\delta\, dx\,dt+\\
\int_{Q_T}  4u^m(u+\gamma)^\eta u_x\zeta^3\zeta_xl_\delta \, dx\,dt
:=A^{(1)}_1+A^{(2)}_1.
\end{multline}
Since $m-\frac{n-4}3+\eta-1>0,$ it follows from Young's inequality
and {Lemma 6.4} that for any $\epsilon>0,$
\begin{multline}\label{5.3}
|A^{(1)}_1|\leq
\epsilon\int_{Q_T\cap\{u>0\}}u_x^6u^{n-4}\zeta^6 \, dx\,dt +C_3(\epsilon)
\int_{Q_T}(u+\gamma)^{\frac32(m-\frac{n-4}3+\eta-1)}\zeta^3\, dx\,dt \leq\\
\epsilon\Biggl( \int_{Q_T\cap\{u>0\}}\zeta^6u^nu_{xxx}^2 \, dx\,dt +C_4
\int_{Q_T}\zeta_x^6u^{n+2} \, dx\,dt \Biggr)+\\
C_3(\epsilon)
\int_{Q_T}(u+\gamma)^{\frac{3m+3\eta+1-n}2}\zeta^3\, dx\,dt.
\end{multline}
With regard to $A^{(2)}_1,$ we have by Young's inequality
\begin{equation}\label{5.4}
|A^{(2)}_1|\leq
|A^{(1)}_1|+C_{5}\int_{Q_T}(u+\gamma)^{m+\eta+1}|\zeta\zeta_x|^2 \, dx\,dt.
\end{equation}

 Thus all the integrals in \eqref{5.1*} are uniformly bounded
with respect to {the} parameters $\delta>0,\gamma>0$. Therefore,
collecting the  estimates obtained for the terms $A_i, i=1,\dots,3,$
passing to  the limit $\delta\to0,$ and then to the limit
$\gamma\to0$, the estimate \eqref{5.1} follows.
\end{proof}
\begin{Theorem}\label{t:5.1}
Let  $u_0$ satisfy \eqref{ic}, \eqref{4.4},  let $\nu=\pm 1,$ $0
\le A,$ $1/2<n<3,$ $m > n/2,$ $m<M$ if $A>0,$ and $m<n+2$ if
$\nu=1,$ $A=0,$ and let $u$ denote an arbitrary strong nonnegative
solution of problem $(\mathbb P)$ obtained as in Theorem 1, which
satisfies the local entropy estimate in Theorem 2 or 3. Then $u$
possesses the finite speed of propagation property in the sense of
Theorem \ref{t:4.1}.
\end{Theorem}
\begin{proof}
Let us consider the  local energy estimate \eqref{62} obtained in
Theorem \ref{Th.le}, and estimate the third term on the right hand
side, setting $\nu=1$ for simplicity,
\begin{multline*}
B:=\int_{Q_T\cap\{u>0\}}u^mu_x(u_x\zeta^6)_{xx} \, dx\,dt= \\ \int_{Q_T\cap\{u>0\}}u^mu_xu_{xxx}\zeta^6 \, dx\,dt
+12\int_{Q_T\cap\{u>0\}}u^mu_xu_{xx}\zeta^5\zeta_x \,dx\,dt
+ \\
6\int_{Q_T\cap\{u>0\}}u^mu_x^2(\zeta^5\zeta_x)_x \, dx\,dt:=B_1+B_2+B_3.
\end{multline*}
Using Young's inequality, we obtain that for any $\epsilon>0,$
\begin{multline*}
|B_1|\leq\bigg|\int_{Q_T\cap\{u>0\}} u^{\frac n2}u_{xxx}u_x u^{\frac{(n-4)}6}u^{m-\frac
n2-\frac{n-4}6}\zeta^6 \, dx\,dt\bigg| \leq
\\ \epsilon\int_{Q_T\cap\{u>0\}}
(u^nu_{xxx}^2+u_x^6u^{n-4})\zeta^6 \, dx\,dt+ D_{1}(\epsilon)\int_{Q_T}
u^{3m-2n+2}\zeta^6\, dx\,dt.
\end{multline*}
Here and in the sequel $D_i$ denote constants which may depend on
$m,$ $n,$ $\eta,$ and on $\epsilon$ if indicated, but which are
independent of $\delta$ and $s$. Similarly, we may estimate
\begin{align*}
& B_2\leq\epsilon\int_{Q_T\cap\{u>0\}}
(u_{xx}^3u^{n-1}+u_x^6u^{n-4})\zeta^6 \, dx\,dt +D_{2}(\epsilon)\int_{Q_T}
u^{2m-n+2}|\zeta^2\zeta_x|^2 \, dx\,dt,
\\
& B_3\leq \epsilon\int_{Q_T\cap\{u>0\}}
u_x^6u^{n-4}\zeta^6 \, dx\,dt +D_{3}(\epsilon)\int_{Q_T}
u^{\frac{3m-n+4}2}|(\zeta^3\zeta_x)_x|^{\frac32}\, dx\,dt.
\end{align*}

Using these estimates in \eqref{62}, we obtain that for
$\epsilon>0$ sufficiently small
\begin{multline}\label{5.8}
\int_\Omega
|u_x(x,T)|^2\zeta^6\, dx+\frac{d_{10}}{2}\int_{Q_T}\zeta^6[(u^{\frac{n+2}6})^6_x +(u^{\frac{n+2}3})_{xx}^3+
(u^{\frac{n+2}2})_{xxx}^2]\, dx\,dt+\\ \frac{d_{10}}{2}\int_{Q_T\cap\{u>0\}}\zeta^6u^nu_{xxx}^2\, dx\,dt
\leq \\ \int_\Omega |u_{0x}|^2\zeta^6 \, dx
+ d_{9}\int_{Q_T}(|\zeta_x|^6+|\zeta\zeta_{xx}|^3)|u|^{n+2}\, dx\,dt +\\
D_{4}\int_{Q_T}[u^{3m-2n+2}\zeta^6 +u^{2m-n+2}|\zeta^2\zeta_x|^2+
u^{\frac{3m-n+4}2}|(\zeta^3\zeta_x)_x|^{\frac32}]\, dx\,dt.
\end{multline}
Assuming that $\eta> \frac{1-n}{3},$  summing the inequalities
\eqref{5.8}, \eqref{5.1}, and taking  $\epsilon>0$ to be
sufficiently small,
 we obtain:
\begin{multline}\label{5.9}
\int_\Omega |u(x,T)|^{\eta+1}\zeta^4 \, dx+\int_\Omega
|u_x(x,T)|^2\zeta^6 \, dx+
\frac{d_{10}}{4}\int_{Q_T\cap\{u>0\}}\zeta^6|(u^{\frac{n+2}2})_{xxx}|^2\, dx\,dt\\\leq
\int_\Omega|u_{0x}|^2\zeta^6 \, dx+\int_\Omega |u_0|^{\eta+1}\zeta^4\, dx+
D_{5}R,\\R:=
\int_{Q_T}u^{3m-2n+2}\zeta^6\, dx\,dt+\int_{Q_T}u^{\frac{3m+3\eta+1-n}2}\zeta^3\, dx\, dt+\\
\int_{Q_T\cap\supp\zeta}u^{n+3\eta-1}\, dx\,dt+
\int_{Q_T}(|\zeta_x|^6+|\zeta\zeta_{xx}|^3)u^{n+2}\, dx\,dt+\\ \int_{Q_T}
u^{2m-n+2}|\zeta^2\zeta_x|^2\, dx\,dt+
\int_{Q_T}u^{n+2\eta}|\zeta_x|^2\, dx\,dt+\\
\int_{Q_T}u^{m+\eta+1}|\zeta
\zeta_x|^2 \, dx\,dt+
\int_{Q_T}
u^{\frac{3m-n+4}2}|(\zeta^3\zeta_x)_x|^{\frac32}\, dx\,dt.
\end{multline}
Let us take  $\zeta(x)$ as the function $\zeta_{s,\delta}$ from
\eqref{4.13}. It then  follows from \eqref{5.9} that
\begin{multline}\label{5.10}
\sup_{t\in(0,T)}\int_{\Omega(s+\delta)}|u(x,t)|^{\eta+1}dx+\sup_{t\in(0,T)}
\int_{\Omega(s+\delta)}|u_x(x,t)|^2\,dx\\+
\frac{d_{10}}{4}\int_{Q_T(s+\delta)}|(u^{\frac{n+2}2})_{xxx}|^2\, dx\,dt\leq
\int_{\Omega(s)}(|u_{0x}|^2 \, dx+|u_0|^{\eta+1})\,dx+D_{6}\tilde R,\\
\tilde R:=
 \int_{Q_T(s)}
u^{3m-3n+2} \,dx\,dt+\int_{Q_T(s)} u^{\frac{3m+3\eta+1-n}2}\, dx\,dt +\int_{Q_T(s)}
u^{n+3\eta-1}\, dx\,dt+\\\delta^{-6}\int_{Q_T(s)}u^{n+2}\, dx\,dt+\\
\delta^{-2}\int_{Q_T(s)}u^{2m-n+2}\,dx\,dt+\delta^{-2}
\int_{Q_T(s)}u^{n+2\eta}\, dx\,dt+\delta^{-2}\int_{Q_T(s)}u^{m+\eta+1}\, dx\,dt+\\\delta^{-3}
\int_{Q_T(s)}u^{\frac{3m-n+4}2}\,dx\,dt:=
\sum_{i=1}^8\delta^{-\chi_i}
\int_{Q_T(s)}u^{\xi_i}\, dx\, dt.
\end{multline}

{We shall use the inequality (\ref{5.10}) to derive a  system
of functional relationships, which will allow us to implement the Stampacchia
Lemma, Lemma 6.2.} {In order to undertake similar computations to those undertaken in
\S 4, it suffices to guarantee that the following inequalities hold:}
\begin{equation}\label{5.11}
\xi_i>1+\eta,\qquad i=1,2,\ldots,8.
\end{equation}

First we  ensure that
\begin{equation}\label{5.12}
\xi_3=n+3\eta-1>1+\eta\Leftrightarrow\eta>1-\frac n2:=\eta_{min}.
\end{equation}
Next we deduce a restriction {on} $m$ by considering
\begin{equation}\label{5.13}
\xi_1=3m-2n+2>1+\eta\Leftrightarrow\frac32(2m-n)+1+\Bigl(1-\frac
n2\Bigr)>1+\eta.
\end{equation}
Together, \eqref{5.12} and \eqref{5.13} yield that
\begin{equation}\label{5.14}
\eta_{min}=1-\frac n2<\eta<\eta_{min}+\frac32(2m-n).
\end{equation}
There exists $\eta$ satisfying (\ref{5.14}) iff
\begin{equation}\label{5.15}
2m-n>0.
\end{equation}
Next it is easy to see that
$$
\xi_2=\frac{3m+3\eta+1-n}2=\frac{(n+3\eta-1)+(3m-2n+2)}2.
$$
Therefore the inequality
\begin{equation}\label{5.16}
\xi_2>1+\eta
\end{equation}
follows from \eqref{5.12} and \eqref{5.13}. It is  easy to check
that the other inequalities in \eqref{5.11} can be satisfied by an
appropriate choice of $\eta$ if conditions \eqref{5.12} and
\eqref{5.13} are satisfied.

As result of \eqref{5.11}, the following Gagliardo--Nirenberg
interpolation inequalities {hold} for $i=1,2,\ldots,8,$
\begin{multline}\label{5.17}
\int_{\Omega(s+2\delta)}u^{\xi_i}dx\leq
D_{7}\Biggl(\int_{\Omega(s+\delta)}\bigl|
\bigl((\zeta_{s+\delta,\delta}^\frac6{n+2}u)^{\frac{n+2}2}\bigr)_{xxx}\bigr|^2dx
\Biggr)^{\frac{\theta_i\xi_i}{n+2}}\\\times \Biggl(
\int_{\Omega(s+\delta)}(\zeta_{s+\delta,\delta}^{\frac6{n+2}}u)^{\eta+1}dx\Biggr)
^{\frac{(1-\theta_i)\xi_i}{\eta+1}},
\end{multline}
where $\theta_i=\frac{(n+2)(\xi_i-\eta-1)}{\xi_i(5\eta+n+7)},$
$i=1,2,\ldots,8.$ We now wish  to guarantee  that
\begin{equation} \label{5.17a}
\frac{\theta_i\xi_i}{n+2}<1, \quad i=1,2,\ldots,8.
\end{equation}
From the definition of $\theta_i,$  it follows that (\ref{5.17a})
holds iff
\begin{equation}\label{5.18}
\eta>\frac{\xi_i-n-8}6,\quad i=1,2,\ldots,8.
\end{equation}
It is easy to check that  all the inequalities in (\ref{5.11}) and
(\ref{5.18}) hold for some $\eta$ in the interval \eqref{5.14} if
condition \eqref{5.15} is satisfied.

Therefore we  deduce from \eqref{5.10}, \eqref{5.17}  the
following inequalities
\begin{multline}\label{5.19}
\int_{Q_T(s+\delta)}u^{\xi_i}\, dx\,dt\leq
D_{8}T^{1-\frac{\theta_i\xi_i}{n+2}}\biggl(\sum_{j=1}^8\delta^{-\chi_j}
\int_{Q_T(s)}u^{\xi_j}\,
dx\,dt+H_0(s)\biggr)^{1+\frac{6(\xi_i-\eta-1)}{5\eta+n+7}},
\end{multline}
 where
$H_0(s):=\int_{\Omega(s)}(u_0^{1+\eta}+|u_{0x}|^2)\,dx=0,$
 for all $0<s<s+\delta<a.$
  From \eqref{5.19} and Lemma 6.2,  the
conclusion of Theorem \ref{t:5.1} now follows.
\end{proof}

\begin{Remark} The method used here for the qualitative analysis of
conditions which guarantee the finite speed of propagation property, should be possible to modify, modulo
various technical difficulties,  to yield
quantitative predictions for propagation rates, {\cite{GS,HSh,Sh2}.   We hope
to perform such an analysis in} a forthcoming paper.
\end{Remark}

\section{Appendix A}

\setcounter{equation}{0}

\begin{Lemma}\label{L.6.1}
(\textbf{Stampacchia's  Lemma \cite{St}}). Assume that a given non-negative,
nonincreasing  function $g:[0, \infty)\to R^1_+$ satisfies
$$
g(s+\delta)\leqslant c_0 (\delta^{-\alpha}g(s))^\beta,\quad
\forall s>0,\,\,\delta>0,
$$
for some real numbers $c_0>0$, $\alpha>0$, $\beta>1$. Then
$$
g(s_0)=0, \quad
s_0:=2^{\frac{\beta}{\beta-1}}(c_0\,g(0)^{\beta-1})^{\frac{1}{2
\beta}}.
$$
\end{Lemma}

\bigskip
For  given $(\beta_1, \cdots, \beta_k) \in \Bbb R^k$, we denote
$$
\beta = \prod^k_{j=1} \beta_j, \quad
\bar\beta_i=\frac{\beta}{\beta_i}=\prod^k_{j=1, j\neq i}\beta_j.
$$

\begin{Lemma}\label{L.6.2} (\textbf{A Stampacchia  Lemma for Systems \cite{GS}}). Assume that $k$
given nonnegative non-increasing functions,  $g_i: [0, \infty)\to [0,
\infty)$, $i=1, \cdots , k$, satisfy {
$$
{g_i(s+\delta)}\leqslant c_i \left( \sum^k_{i=1}
\delta^{-\alpha_i} g_i(s)\right)^{\beta_i},\quad \forall s>0,
\,\,\delta>0, \,\,i=1, \cdots , k,
$$
for some real numbers} $c_i>0$, $\beta_i>1$; $\alpha_i\geqslant 0$,
$i=1,\cdots, k$, and $\alpha_i>0$ for $i=1, \cdots, l$,
$l\geqslant 1$. Denote
$$
g(s)=\sum^k_{i=1}\left(\overline{
c_i^{\bar\beta_i}}\right)\left(g_i(s)\right)^{\bar\beta_i},
$$
and assume that the function $Q(s)$ defined by
$$
Q(s)=
\begin{cases}
k^\beta \sum\limits^k_{i=l+1} c^{\bar\beta_i}_i
\left(\overline{c_i^{\bar\beta_i}}
\right) ^{1-\beta_i} g(s)^{\beta_i-1},& \text{if}\quad l<k, \\
0, &\text{if}\quad l=k,
\end{cases}
$$
satisfies the constraint $Q(s_1)<1$ at some point $s_1\geqslant 0$. Then there exists a
positive constant $C>1$, which depends on $k, l, \alpha_i, \beta_i,$
and $Q(s_1)$,  such that
$$
g_i(s_0)=0,\quad \forall \quad i=1, \cdots, l,
$$
where
$$
s_0:= s_1+C \sum^l_{i=1}\left(c^{\bar\beta_i}_i
(\overline{c^{\bar\beta_i}_i})^{1-\beta_i} (g(s_1))^{\beta_i-1}
\right)^{\frac{1}{\alpha_i \beta}}.
$$
\end{Lemma}

\begin{Remark} If $l=k$, then $s_1=0$ in Lemma \ref{L.6.2}\end{Remark}

\begin{Lemma}\label{L.6.3}(\textbf{Generalized Bernis inequalities for periodic functions \cite{GS}}). Let $\Omega=(a,b)\subset\Bbb
R^1$, $n\in (\frac{1}{2},3)$, and let $u\in C^1(\bar \Omega)\cap
H^3_\text{loc}(\{u>0\})$ be an  arbitrary $|\Omega|$  periodic
function such that
$$
u\geqslant 0,\qquad u_x\in L^2(\Omega), \qquad \int_{\{u>0\}}
u^n|u_{xxx}|^2\,dx<\infty.
$$
Then $u^{\frac{n+2}{6}}\in W^{1,6}(\Omega)$, $u^{\frac{n+2}{3}}\in
W^{2,3}(\Omega)$, $u^{\frac{n+2}{2}}\in W^{3,2}(\Omega):=H^3(\Omega)$,
and there exists a positive constant, $C>1$, which depends on $n$ only,  such
that
$$
\int_\Omega
\zeta^6\left(\left|\left(u^{\frac{n+2}{6}}\right)_x\right|^6 +
\left|\left(u^{\frac{n+2}{3}}\right)_{xx}\right|^3 +
\left|\left(u^{\frac{n+2}{2}}\right)_{xxx}\right|^2\right)\,dx
$$
$$
\leqslant C \int_{\{u>0\}}\zeta^6 u^n|4_{xxx}|^2\,dx+ C
\int_\Omega|\zeta_x|^6 u^{n+2}\,dx.
$$
\end{Lemma}

\bigskip
\par\noindent{\bf Acknowledgment.} {A. N.-C.} would like to acknowledge the support of the M.R. Saulson Research Grant.
{A. Sh.} would like to acknowledge the hospitality and support
of the Math Department of the Technion during visits to Haifa,
Israel.

\end{document}